\documentclass[12pt, a4paper]{amsart}
\usepackage{amsmath}
\usepackage{geometry,amsthm,graphics,tabularx,amssymb,color,verbatim}
\usepackage{amscd}
\usepackage{bm,bbm}
\usepackage[all]{xypic}

\newcommand{\wh}{\widehat}
\newcommand{\Res}{\mathrm{Res}}

\newcommand{\ot}{\otimes}
\newcommand{\op}{\oplus}
\newcommand{\CC}{\mathcal{C}}

\newcommand{\CA}{\mathcal{A}}

\newcommand{\g}{\mathfrak{g}}

\newcommand{\fgl}{\mathfrak{gl}}
\newcommand{\C}{\mathbb{C}}
\newcommand{\D}{\mathcal{D}}

\newcommand{\N}{\mathbb N}
\newcommand{\Z}{\mathbb Z}

\newcommand{\Hom}{\mathrm{Hom}}

\newcommand{\End}{\mathrm{End}}

\newcommand{\U}{\mathcal{U}}

\newcommand{\te}[1]{\textnormal{{#1}}}

\theoremstyle{Theorem}

\theoremstyle{Theorem}

\newtheorem{thm}{Theorem}[section]

\newtheorem{lemt}[thm]{Lemma}
\newtheorem{prpt}[thm]{Proposition}

\newtheorem{remt}[thm]{Remark}

\newtheorem{dfnt}[thm]{Definition}

\def\({\left(}

\def\){\right)}

\newlength{\dhatheight}

\def \<{{\langle}}
\def \>{{\rangle}}

\newcommand{\vac}{\mathbf{1}}
\newcommand{\E}{{\mathcal{E}}}

\allowdisplaybreaks

\numberwithin{equation}{section}
\title[ $(G,\chi_\phi)$-equivariant $\phi$-coordinated modules for vertex algebras]{$(G,\chi_\phi)$-equivariant $\phi$-coordinated quasi modules for vertex algebras}

\author{Fulin Chen$^1$}
\address{School of Mathematical Sciences, Xiamen University,
 Xiamen, China 361005} \email{chenf@xmu.edu.cn}
 \thanks{$^1$Partially supported by China NSF grant (No.11971397) and the Fundamental
Research Funds for the Central Universities (No.20720190069).}
 \author{Xiaoling Liao}\address{School of Mathematical Sciences, Xiamen University,
 Xiamen, China 361005} \email{xiaoling@stu.xmu.edu.cn}
 \author{Shaobin Tan$^2$}\address{School of Mathematical Sciences, Xiamen University,
 Xiamen, China 361005} \email{tans@xmu.edu.cn }\thanks{$^2$Partially supported by China NSF grants (No.11971397)}
\author{Qing Wang$^3$}\address{School of Mathematical Sciences, Xiamen University,
 Xiamen, China 361005} \email{qingwang@xmu.edu.cn }\thanks{$^3$Partially supported by
 China NSF grants (No.12071385) and the Fundamental Research Funds for the Central Universities (No.20720200067).}


\begin{document}
\bibliographystyle{plain}

\begin{abstract}To give a unified treatment on the association of Lie algebras and vertex algebras, we study $(G,\chi_\phi)$-equivariant $\phi$-coordinated quasi modules for vertex algebras,
where $G$ is a group with $\chi_\phi$ a linear character of $G$ and  $\phi$ is an associate of the
one-dimensional additive formal group. The theory of $(G,\chi_\phi)$-equivariant $\phi$-coordinated quasi modules for nonlocal vertex algebra is established in \cite{JKLT}. In this paper, we concentrate on the context of vertex algebras. We establish  several conceptual results, including a
generalized commutator formula and a general construction of vertex algebras and their $(G,\chi_\phi)$-equivariant $\phi$-coordinated quasi modules.
Furthermore, for any conformal algebra $\CC$, we construct a class of Lie algebras $\wh{\CC}_\phi[G]$  and prove that
 restricted $\wh{\CC}_\phi[G]$-modules are exactly $(G,\chi_\phi)$-equivariant $\phi$-coordinated quasi modules for the
 universal enveloping vertex algebra of $\CC$.
As an application, we determine the $(G,\chi_\phi)$-equivariant $\phi$-coordinated quasi modules for  affine and Virasoro vertex algebras.
\end{abstract}
\maketitle
\section{Introduction}
Modules and twisted modules have been extensively studied in vertex algebra theory.  We usually associate Lie algebras with vertex algebras through modules or twisted modules (see \cite{FLM,FZ,Li1,Li6}, etc). But these theories are not valid for some algebras which generating functions are not ``local", we need generalize the notion of modules.  A notion of quasi module of vertex algebra was introduced in \cite{Li5} to associate more infinite dimensional Lie algebras like quantum torus Lie algebras with vertex algebras.
   In order to associate quantum affine algebras with quantum vertex algebras, a theory of $\phi$-coordinated module for the quantum vertex algebra was developed in \cite{Li2}. The parameter $\phi$ is an associate of the one-dimensional additive formal group $F(x,y)=x+y$.
When $\phi=x+z$, a $\phi$-coordinated module of vertex algebra is just a module.
To establish the explicit connections of quantum vertex algebras to various algebraic structures like quantum affine algebras, a notion of  $(G,\chi_\phi)$-equivariant $\phi$-coordinated quasi module for nonlocal vertex algebras was introduced in \cite{JKLT}, where $G$ is a group and
$\chi_\phi$ is a linear character of $G$.
 In \cite{Li2}, the author proved that for every $p(x)\in \C((x))$, $\phi=e^{zp(x)\frac{d}{dx}}(x)$ is an associate
of $F(x,y)$ and every associate is of this form.
Moreover, an explicit association between quantum vertex algebra and certain quantum $\beta\gamma$-system was established in \cite{Li2} through  $\phi$-coordinated modules, where $\phi=xe^z$, i.e., $p(x)=x$.
Recently, the $\phi_\epsilon$-coordinated modules of vertex algebras have been studied precisely in \cite{BLP} for $p(x)=x^\epsilon$ with $\epsilon$ an integer, and they use the theory of $\phi_\epsilon$-coordinated modules to associate vertex algebras with Novikov algebras. Specifically, for $\phi=xe^z$, a notion of  $G$-equivariant $\phi$-coordinated quasi module for quantum vertex algebras was introduced in \cite{Li4}. A commutator formula for  $G$-equivariant $\phi$-coordinated quasi module was obtained and a natural connection between the
deformed Virasoro algebra $Vir_{p,-1}$ of  a certain Clifford vertex superalgebra  through $G$-equivariant $\phi$-coordinated quasi module was established therein. Later in  \cite{JKLT}, a generalized commutator formula and a general construction of $(G,\chi_\phi)$-equivariant $\phi$-coordinated quasi modules for nonlocal vertex algebras were established, and by using these results, they give a construction of  (equivariant) $\phi$-coordinated quasi modules for lattice vertex algebras. We note that if $\phi=xe^z$, the notion of $(G,\chi_\phi)$-equivariant $\phi$-coordinated quasi module coincides with the notion of  $G$-equivariant $\phi$-coordinated quasi module.  Motivated by these works, in this paper,  we focus on the frame of vertex algebra.  We expect to unify the association of Lie algebras and vertex algebras though $(G,\chi_\phi)$-equivariant $\phi$-coordinated quasi modules.

 Next, we briefly outline our ideas. Let $\phi=e^{zp(x)\frac{d}{dx}}(x)$ for $0\neq p(x)\in\C((x))$. A $\phi$-coordinated module for a  vertex algebra $V$ is a vector space $W$ with a linear map $Y_W(\cdot,z)$ from $V$ to $\Hom(W,W((z)))$ satisfying the conditions that $Y_W(\vac,z)=1_W$ and that for $u,v\in V$, there exists a nonnegative integer $k$ such that
  \begin{eqnarray*}
&&(z_1-z_2)^kY_{W}(u,z_{1})Y_{W}(v,z_{2})\in\Hom(W,W((z_{1},z_{2}))),\\
&&(\phi(z_{2},z_{0})-z_2)^k Y_{W}(Y(u,z_{0})v,z_{2})=((z_1-z_2)^k Y_{W}(u,z_{1})Y_{W}(v,z_{2}))\mid_{z_{1}=\phi(z_{2},z_{0})}.
\end{eqnarray*} By generalizing the corresponding result in \cite{BLP}, we prove that if $(W,Y_W)$ is a $\phi$-coordinated $V$-module, then for $u,v\in V$, we have the commutator formula
  \begin{equation*}\begin{split}
[Y_{W}(u,z_{1}),Y_{W}(v,z_{2})]=\sum_{i\geq0}\frac{1}{i!}Y_{W}(u_{i}v,z_{2})\left(p(z_2)\frac{\partial}{\partial z_{2}}\right)^{i}p(z_1)z_{1}^{-1}\delta\left(\frac{z_{2}}{z_{1}}\right).
\end{split}\end{equation*}
 For any vector space $W$, we prove that any (quasi) local subset of $\E(W)(=\Hom(W,W((z))))$ can generate a vertex algebra with $W$ is naturally a $\phi$-coordinated (quasi) module of it.
To give a unified treatment on the association of Lie algebras and vertex algebras, we need the notion of conformal algebra which was introduced in \cite{K} (also known as a vertex Lie algebra in \cite{DLM,P}) as a bridge. A conformal algebra $\CC$ is a $\C[\partial]$-module with a $\C$-bilinear product
$a_nb$ for each $n\in\N$ satisfying some axioms.
Define a multiplication on the formal loop space $\CC\ot \C((x))$ by
 \begin{eqnarray*}
 [a\ot f(x),b\ot g(x)]=\sum_{i\geq0}\frac{1}{i!}(a_ib)\ot \left(\left(p(x)\frac{d}{d x}\right)^{i}f(x)\right)g(x)
\end{eqnarray*}
for  $a,b\in\CC$ and $f(x),g(x)\in\C((x))$.
Then the quotient space
\begin{align*}
\wh{\CC}_\phi:=\CC\ot \C((x))/\mathrm{Im}\, (\partial \ot 1+1\ot p(x)\frac{d}{d x})
\end{align*}
with this multiplication form a Lie algebra \cite{K}. For $a\in\CC$, we set
\[a_{\phi}(z)=\sum_{n\in\Z}\overline{a\ot x^np(x)}z^{-n-1}\in \wh{\CC}_\phi[[z,z^{-1}]],\]
where $\overline{a\ot x^np(x)}$ is the image of $a\ot x^np(x)$ in $\wh{\CC}_\phi$. Then for $a,b\in \CC$,
\begin{eqnarray*}
[a_{\phi}(z),b_{\phi}(w)]=\sum_{i\geq0}\frac{1}{i!}(a_ib)_{\phi}(w) \left(p(w)\frac{\partial}{\partial w}\right)^{i}p(z)z^{-1}\delta\left(\frac{w}{z}\right).
\end{eqnarray*}
We prove that restricted $\wh{\CC}_\phi$-modules are exactly $\phi$-coordinated  modules for the
 universal enveloping vertex algebra $V_\CC$ of $\CC$.

Let $G$ be a group with a linear character $\chi$. Following \cite{JKLT}, a {\em $(G,\chi)$-vertex algebra}
is a vertex algebra $V$ with a $G$-action $R$ on $V$ such that $R_g(\vac)=\vac$ and
\[R_gY(v,z)R_g^{-1}=Y(R_g(v),\chi(g)z)\quad\te{for }g\in G,v\in V.\] The notion of $(G,\chi)$-vertex algebra was introduced in \cite{Li3} (cf. \cite{Li5}).
A $(G,\chi_\phi)$-equivariant $\phi$-coordinated quasi module is defined for a $(G,\chi)$-vertex algebra and from \cite{JKLT}, we may always assume that $$\phi(x,\chi(g)z)=\chi_\phi(g)\phi(\chi_\phi(g)^{-1}x,z)\quad \te{for }g\in G.$$ In this paper, we obtain  a general construction of vertex algebras and their $(G,\chi_\phi)$-equivariant $\phi$-coordinated quasi modules based on \cite{JKLT}. Let $(W,Y_W)$ be a $(G,\chi_\phi)$-equivariant $\phi$-coordinated quasi $V$-module, we prove the following generalized commutator formula:
\begin{equation*}\begin{split}
&[Y_W(u,z_1),Y_W(v,z_2)]\\
=\Res_{z_0}&\sum_{g\in \psi(\chi_\phi(G))}Y_W(Y(R_{g^{-1}}u,z_0)v,z_2)e^{z_0p(z_2)\frac{\partial}{\partial z_2}}\left(\chi(g) p(z_1)z_1^{-1}\delta\left(\frac{\chi_\phi(g)z_2}{z_1}\right)\right),
\end{split}\end{equation*}
 where $\psi:\chi_\phi(G)\rightarrow G$ is a section of $\chi_\phi$. The notion of $(G,\chi)$-conformal algebra was introduced in  \cite{Li3} which generalizes the notion of $\Gamma$-conformal algebra in \cite{G-K-K}. A {\em $(G,\chi)$-conformal algebra} is a conformal
 algebra $\CC$
 with a $G$-action on $\CC$ satisfying $$R_g(a_{m}b)=\chi(g)^{-(m+1)}(R_ga)_{m}(R_gb)\quad\te{ for any }a,b\in\CC, m\in \N,$$ and some other compatible conditions.
 Let $(\CC,R)$ be a $(G,\chi)$-conformal algebra,
for any $g\in G$, we define a linear automorphism $\hat R_g$ on $\CC\ot \C((x))$ by
\begin{align*}
\hat R_g(a\ot f(x))=\chi(g)^{-1}R_g(a)\ot f(\chi_\phi(g)^{-1}x),\quad a\in \CC, f(x)\in \C((x)).
\end{align*} Then we define a multiplication on $ \wh\CC_\phi$ by
\begin{align*}
(u,v)\mapsto \sum_{g\in G} [\hat R_g(u),v],\quad u,v\in \wh\CC_\phi.
\end{align*} It follows that the
quotient space
\begin{align*}
\wh\CC_\phi[G]:=\wh\CC_\phi/\te{span}\{\hat R_g(u)-u\mid g\in G, u\in \wh\CC_\phi\}
\end{align*}
is a Lie algebra  with  this new multiplication.
For $a\in \CC$, $f(x)\in \C((x))$, we still denote the image of $\overline{a\ot f(x)}$ in $\wh\CC_\phi[G]$ by $\overline{a\ot f(x)}$.
For $a\in\CC$, we set
\[a^G_{\phi}(z)=\sum_{n\in\Z}\overline{a\ot x^np(x)}z^{-n-1}\in \wh\CC_\phi[G][[z,z^{-1}]].\]
Then for $a,b\in \CC$, we have
\begin{equation*}\begin{split}
&[a^G_{\phi}(z),b^G_{\phi}(w)]\\
=&\sum_{g\in G}\sum_{i\geq 0}\frac{1}{i!}\left((R_{g^{-1}}a)_{i}b\right)^G_{\phi}(w)\left(p(w)\frac{\partial}{\partial w}\right)^{i}\left(\chi(g) p(z)z^{-1}\delta\left(\frac{\chi_\phi(g)w}{z}\right)\right).
\end{split}\end{equation*}

Let $H=\ker\chi_\phi$ and let $\bar\chi_\phi$ be the linear character of $G/H$ induced by $\chi_\phi$. From \cite{Li3}, the quotient space  \[\CC/H=\CC/\te{span}\{R_h a-a\mid h\in H,\ a\in\CC\}\]
has a natural $(G/H,\bar{\chi})$-conformal algebra structure, where $\bar{\chi}$ is a linear character of $G/H$ induced by $\chi$. Then we prove that restricted $\widehat{\CC}_{\phi}[G]$-modules are exactly $(G/H,\bar\chi_\phi)$-equivariant
   $\phi$-coordinated quasi $V_{\CC/H}$-modules.
Finally, we apply these theories on  affine and Virasoro vertex algebras, and the vertex algebras associated
to Novikov algebras.

This paper is organized as follows. In Section 2, we recall $\phi$-operations on $\E(W)$ and establish some technical lemmas. In Section 3, we recall the notion of $\phi$-coordinated (quasi) module
for a vertex algebra and establish some conceptual results. In Section 4, we prove that
 restricted $\wh{\CC}_\phi$-modules are exactly $\phi$-coordinated $V_\CC$-modules.  In Section 5, we establish a
generalized commutator formula and a general construction of  $(G,\chi_\phi)$-equivariant $\phi$-coordinated quasi modules of vertex algebras. Furthermore,  we construct a class of Lie algebras $\wh{\CC}_\phi[G]$ and prove that restricted $\wh{\CC}_\phi[G]$-modules are exactly $(G/H,\bar\chi_\phi)$-equivariant $\phi$-coordinated quasi $V_{\CC/H}$-modules. In Section 6,  we determine a class of $(G,\chi_\phi)$-equivariant $\phi$-coordinated quasi modules for affine vertex algebras, and determine the $\phi$-coordinated modules for
Virasoro vertex algebra and the vertex algebras associated to Novikov algebras.

In this paper, we let $\Z$, $\Z^{\times}$, $\Z_+$, $\N$ and $\C$ be the sets of integers, nonzero integers, positive integers, nonnegative integers and  complex numbers,
respectively.  All the
Lie algebras in this paper are over the field of complex numbers.
For any linear endomorphism  $\varphi$ of a vector space,  we set $\varphi^{(k)}=\frac{\varphi^k}{k!}$.

\section{$\phi$-operations on $\E(W)$}
In this section, we recall the $\phi$-operations following \cite{Li2} and prove some technical Lemmas. We begin by recalling  the notion of an associate for the one-dimensional additive formal group (law)  $F(x,y)=x+y$.
Following \cite{Li2}, an associate $\phi(x,z)$ (or simply $\phi$) of $F(x,y)$  by definition is a formal series in $\C((x))[[z]]$ satisfying
$$\phi(x,0)=x,\quad  \phi(\phi(x,z_1),z_2)=\phi(x,z_1+z_2).$$
It was proved therein that  for every $p(x)\in \C((x))$, $\phi(x,z)=e^{zp(x)\frac{d}{dx}}(x)$ is an associate
of $F(x,y)$ and every associate is of this form with $p(x)$ uniquely determined.
Throughout this paper, we fix an associate
\begin{align}\label{defphi}
\phi=\phi(x,z)=e^{zp(x)\frac{d}{dx}}(x)\quad\te{with}\quad 0\ne p(x)\in \C((x)).\end{align}

Throughout this section let $W$ be a fixed vector space and set \[\E(W)=\Hom(W,W((z)))\subset(\End W)[[z,z^{-1}]].\]
A pair $(a(z),b(z))$ in $\E(W)$ is said to be {\em quasi local} if
there exists a nonzero polynomial $q(z_1,z_2)\in\C[z_1,z_2]$ such that
 \begin{align}\label{eq:quasiloal}
 q(z_1,z_2)\ [a(z_1),b(z_2)]=0,\end{align}
 and is said to be {\em local} if the polynomial $q(z_1,z_2)$ is assumed to a polynomial of the form $(z_1-z_2)^n$ for some $n\in \N$.
The main goal of this section is to establish a commutator formula for quasi local pairs in $\E(W)$  in terms of the following $\phi$-operations  introduced in \cite{Li2}:
\begin{dfnt}\label{de:yephi}
{\em Let $(a(z),b(z))$ be a quasi local pair in $\E(W)$ such that \eqref{eq:quasiloal} holds.
We define
\[Y_\E^\phi(a(z),z_0)b(z)=\sum\limits_{n\in\Z}a(z)_{n}^\phi b(z)z_0^{-n-1}\in\E(W)((z_0))\]
by the following rule:
\[Y_\E^\phi(a(z),z_0)b(z)=q(\phi(z,z_0),z)^{-1}(q(z_1,z)a(z_1)b(z))\mid_{z_1=\phi(z,z_0)}.\]}
\end{dfnt}

It was proved therein that the definition of $Y_\E^\phi$ dose not depend on the choice of $q(z_1,z_2)$.
In what follows, we establish   some technical lemmas.

\begin{lemt}\label{delta-1}
Let $j,k$ be two nonnegative integers with $j>k$. Then for any nonzero complex number $\lambda$, we have
\begin{align}\label{eq:lemadelta-1}
(z_1-\lambda z_2)^{j}\left(p(z_{2})\frac{\partial}{\partial z_{2}}\right)^{(k)}\delta\left(\frac{\lambda z_{2}}{z_{1}}\right)=0.\end{align}
\end{lemt}
\begin{proof}
Note that we may write
$$\left(p(z_{2})\frac{\partial}{\partial z_{2}}\right)^{(k)}=f_k(z_2)\left(\frac{\partial}{\partial z_{2}}\right)^{(k)}+
\cdots+f_{1}(z_{2})\frac{\partial}{\partial z_{2}}+f_0(z_2)$$
for some uniquely determined $f_0(z_2),\cdots,f_k(z_2)\in \C((z_2))$.
Then the assertion follows from the fact that $(z_1-\lambda z_2)^{j}\left(\frac{\partial}{\partial z_{2}}\right)^{(i)}\delta\left(\frac{\lambda z_{2}}{z_{1}}\right)=0$
for  $i=0,\cdots,k$.
\end{proof}

For convenience, we set
\[\bar{p}(x)=x^{-1}p(x)\in \C((x)).\]

\begin{lemt}\label{delta-2} Let $\lambda_1,\cdots,\lambda_r$ be distinct nonzero complex numbers, $k_1,\cdots,k_r\in \Z_+$ and
$A_{ij}(z)\in \E(W)$, where $i=1,\cdots,r,$ and for a fixed $i$,
$ 0\le j\le k_i-1$.
   Then
\begin{eqnarray}\label{eq:lemma220}
\sum _{i=1}^r\sum_{j=0}^{k_i-1} A_{ij}(z_{2})\left(p(z_{2})\frac{\partial}{\partial z_{2}}\right)^{(j)}\bar{p}(\lambda_i^{-1}z_1)\delta\left(\frac{\lambda_iz_{2}}{z_{1}}\right)=0
\end{eqnarray}
if and only if $A_{ij}(z)=0$ for all $i,j$.
\end{lemt}
\begin{proof} For $1\le i\le r$, set
\[f_i(z)=\frac{\prod_{j=1}^r (z-\lambda_j)^{k_j}}{(z-\lambda_i)^{k_i}}\in\C[z].\]
In view of Lemma \ref{delta-1}, for $1\le i\ne s\le r$ and $0\le j\le k_s-1$ we have
\begin{align}\label{eq:lemma221}
f_i(z_1/z_2)\left(p(z_{2})\frac{\partial}{\partial z_{2}}\right)^{(j)}\bar{p}(\lambda_s^{-1}z_1)\delta\left(\frac{\lambda_sz_{2}}{z_{1}}\right)=0.
\end{align}
As $f_1(z),\cdots,f_r(z)$ are pairwise relatively prime, we may take suitable polynomials $q_1(z),\cdots,q_r(z)\in \C[z]$ such that
\[q_1(z)f_1(z)+\cdots+q_r(z)f_r(z)=1.\]
This together with \eqref{eq:lemma221} gives that for $1\le s\le r$, $j=0,\cdots,k_s-1$,
\begin{equation}\label{eq:lemma222}
\begin{split}
&A_{sj}(z_{2})\left(p(z_{2})\frac{\partial}{\partial z_{2}}\right)^{(j)}\bar{p}(\lambda_s^{-1}z_1)\delta\left(\frac{\lambda_sz_{2}}{z_{1}}\right)\\
=\,&\(\sum_{i=1}^r f_i(z_1/z_2) q_i(z_1/z_2)\)A_{sj}(z_{2})\left(p(z_{2})\frac{\partial}{\partial z_{2}}\right)^{(j)}\bar{p}(\lambda_s^{-1}z_1)\delta\left(\frac{\lambda_sz_{2}}{z_{1}}\right)\\
=\,&A_{sj}(z_{2})f_s(z_1/z_2)q_s(z_1/z_2)\left(p(z_{2})\frac{\partial}{\partial z_{2}}\right)^{(j)}
\bar{p}(\lambda_s^{-1}z_1)\delta\left(\frac{\lambda_sz_{2}}{z_{1}}\right).
\end{split}
\end{equation}

Now we assume that \eqref{eq:lemma220} holds and fix $s$ with $1\le s\le r$.
Then from \eqref{eq:lemma221} and \eqref{eq:lemma222} we have
 \begin{align*}
0=&\sum _{i=1}^r\sum_{j=0}^{k_i-1} A_{ij}(z_{2})f_s(z_1/z_2)q_s(z_1/z_2)\left(p(z_{2})\frac{\partial}{\partial z_{2}}\right)^{(j)}\bar{p}(\lambda_i^{-1}z_1)\delta\left(\frac{\lambda_iz_{2}}{z_{1}}\right)\\
= &\sum_{j=0}^{k_s-1} A_{sj}(z_{2})f_s(z_1/z_2)q_s(z_1/z_2)\left(p(z_{2})\frac{\partial}{\partial z_{2}}\right)^{(j)}\bar{p}(\lambda_s^{-1}z_1)\delta\left(\frac{\lambda_sz_{2}}{z_{1}}\right)\\
=&\sum_{j=0}^{k_s-1} A_{sj}(z_{2})\left(p(z_{2})\frac{\partial}{\partial z_{2}}\right)^{(j)}\bar{p}(\lambda_s^{-1}z_1)\delta\left(\frac{\lambda_sz_{2}}{z_{1}}\right),
\end{align*}
which amounts to
 \begin{eqnarray}\label{eq:lemma223}
 \sum_{j=0}^{k_s-1} A_{sj}(z_{2})\left(p(z_{2})\frac{\partial}{\partial z_{2}}\right)^{(j)}p(z_1)z_1^{-1}\delta\left(\frac{z_{2}}{z_{1}}\right)=0.
\end{eqnarray}
For $j=0,\cdots,k_s-1$, write
\begin{align*}
A_{sj}(z_{2})\left(p(z_{2})\frac{\partial}{\partial z_{2}}\right)^{(j)}
=\sum_{n=0}^j g_{jn}(z_2)\left(\frac{\partial}{\partial z_{2}}\right)^{(n)}
\end{align*}
for some uniquely determined $g_{jn}(z_2)\in \C((z_2))$.
In view of \eqref{eq:lemma223} we have
$$\sum_{n=0}^{k_s-1} \(\sum_{j=n}^{k_s-1} g_{jn}(z_2)\)\left(\frac{\partial}{\partial z_{2}}\right)^{(n)}z_1^{-1}\delta\left(\frac{z_{2}}{z_{1}}\right)=0.
$$
 By \cite[Lemma 2.1.4]{Li1}, this implies $\sum_{j=n}^{k_s-1} g_{jn}(z_2)=0$ for $0\le n\le k_s-1$,
 and hence implies $A_{sj}(z)=0$ for $0\le j\le k_s-1$, as desired.
\end{proof}

The following result is a slight generalization of  \cite[Lemma 2.3]{Li4}.
\begin{lemt}\label{lem:ab-k} Let $a(z),b(z)\in \E(W)$ and let
\begin{align}\label{eq:qz}
q(z)=(z-\lambda_1)^{k_1}\cdots(z-\lambda_r)^{k_r}\in\C[z],
\end{align} where
$\lambda_1,\cdots,\lambda_r\in \C^\times$ are distinct and  $k_1,\cdots,k_r\in \Z_+$.
Then
\begin{align}\label{eq:lemma230}
q(z_1/z_2)\,[a(z_1),b(z_2)]=0\end{align}
if and only if
\begin{equation}\label{eq:lemma231}\begin{split}
[a(z_1), b(z_2)]=\sum _{i=1}^r\sum_{j=0}^{k_i-1} A_{ij}(z_{2})\left(p(z_{2})\frac{\partial}{\partial z_{2}}\right)^{(j)}\bar{p}(\lambda_i^{-1}z_1)\delta\left(\frac{\lambda_iz_{2}}{z_{1}}\right)
\end{split}\end{equation}
for some uniquely determined $A_{ij}(z)\in\E(W)$.
\end{lemt}
\begin{proof}The proof  is similar to that of  \cite[Lemma 2.3]{Li4} and we give a sketch of it.
The uniqueness of $A_{ij}(z)$ follows from Lemma \ref{delta-2}.
For the existence,  write
\[\frac{1}{q(x)}=\sum _{i=1}^r\sum_{j=1}^{k_i}\frac{a_{ij}}{(x-\lambda_i)^j}\]
for some $a_{ij}\in\C$.
Then from \eqref{eq:lemma230} we obtain (cf. \cite[Lemma 2.3]{Li4})
\begin{equation}\label{eq:lemma232}\begin{split}
[a(z_1), b(z_2)]=\sum _{i=1}^r\sum_{j=0}^{k_i-1} a_{i,j+1}\lambda_i^{-j}z_2^{j+1}A(z_1,z_2)\left(\frac{\partial}{\partial z_{2}}\right)^{(j)}z_1^{-1}\delta\left(\frac{\lambda_iz_{2}}{z_{1}}\right),
\end{split}\end{equation}
where  $A(z_1,z_2)=q(z_1/z_2)a(z_1)b(z_2)\in \Hom(W,W((z_1,z_2)))$.

Note that for any $n\in \N$, there exist $f_0(z_2),\cdots,f_{n}(z_2)\in\C((z_2))$ such that
\begin{equation*}
\left(\frac{\partial}{\partial z_2}\right)^{(n)}
=f_n(z_2)\left(p(z_2)\frac{\partial}{\partial z_2}\right)^{(n)}+f_{n-1}(z_2)\left(p(z_2)\frac{\partial}{\partial z_2}\right)^{(n-1)}+\cdots+f_{0}(z_2).
\end{equation*}
An induction argument shows that for any $B(z_1,z_2)\in\Hom(W,W((z_1,z_2)))$ and $n\in \N$,
there exist $B_0(z),\cdots,B_n(z)\in\E(W)$ such that
\begin{equation*}\begin{split}
&B(z_1,z_2)\left(\frac{\partial}{\partial z_{2}}\right)^{(n)}z_1^{-1}\delta\left(\frac{\lambda_iz_{2}}{z_{1}}\right)=\sum_{i=0}^nB_i(z_2)\left(\frac{\partial}{\partial z_{2}}\right)^{(i)}\bar{p}(\lambda_i^{-1}z_1)\delta\left(\frac{\lambda_iz_{2}}{z_{1}}\right).
\end{split}\end{equation*}
Thus \eqref{eq:lemma231} follows.
\end{proof}

We fix a $\hat{p}(x)\in \C((x))$ such that
 $\frac{d}{dx}\hat{p}(x)=p^{-1}(x)$, and set
\[f_{p(x)}(x,z):=\hat{p}(x(1+z))-\hat{p}(x).\]
Then we can deduce that
\begin{align}\label{hatphi}
\phi(x,f_{p(x)}(x,z))=x(1+z).
\end{align}
We have (cf.\,\cite{BLP,Li2}):
\begin{lemt}\label{lem:formal-jocobi} Let
\begin{eqnarray*}
&&A(z_1,z_2)\in\Hom\big(W,W((z_1))((z_2))\big),\quad B(z_1,z_2)\in\Hom\big(W,W((z_2))((z_1))\big),\\
&&C(z_0,z_2)\in\big(\Hom\big(W,W((z_2))\big)\big)((z_0)).
\end{eqnarray*}
If there exists a nonnegative integer $k$ such that
\begin{eqnarray*}
&&(z_1-z_2)^kA(z_1,z_2)=(z_1-z_2)^k B(z_1,z_2),\\
&&\left((z_1-z_2)^kA(z_1,z_2)\right)|_{z_1=\phi(z_2,z_0)}=(\phi(z_2,z_0)-z_2)^kC(z_0,z_2),
\end{eqnarray*}
then
\begin{equation*}\begin{split}
&(z_{2}z)^{-1}\delta\left(\frac{z_{1}-z_{2}}{z_{2}z}\right)A(z_1,z_2)-(z_{2}z)^{-1}\delta\left(\frac{z_{2}-z_{1}}{-z_{2}z}\right)B(z_1,z_2)\\
&\quad\quad=z_{1}^{-1}\delta\left(\frac{z_{2}(1+z)}{z_{1}}\right)C\left(f_{p(z_2)}(z_2,z),z_2\right).
\end{split}\end{equation*}
\end{lemt}
\begin{proof}
When $p(x)=x^\epsilon$ for some $\epsilon\in \Z$, the assertion
is proved in \cite[Lemma 3.16]{BLP}, where $f_{p(x)}(x,z)$ is denoted as $f_\epsilon(x,z)$.
In view of \eqref{hatphi}, the proof therein also valid for a general $p(x)$.
\end{proof}

Now we are in a position to state the main result of this section, which says that those $A_{ij}(z)$ in \eqref{eq:lemma231} are nothing but
the $\phi$-products $a(\lambda_iz)_j^\phi b(z)$.

\begin{prpt}\label{prop:quasilocalpair} Let $(a(z),b(z))$ be a pair in $\E(W)$ satisfying
\begin{align*} q(z_1/z_2)[a(z_1),b(z_2)]=0
\end{align*}
for some $q(z)\in \C[z]$.
Then
\begin{eqnarray*}
[a(z_1),b(z_2)]=\sum_{\lambda\in \C^\times}\sum_{j\ge 0}a(\lambda z_2)_{j}^{\phi}b(z_2)
\left(p(z_{2})\frac{\partial}{\partial z_{2}}\right)^{(j)}\bar{p}(\lambda^{-1}z_1)\delta\left(\frac{\lambda z_{2}}{z_{1}}\right),
\end{eqnarray*}
which is a finite sum. Furthermore, the coefficients  $a(\lambda z_2)_{j}^{\phi}b(z_2)\ne 0$ only if $\lambda$ is a root
of $q(z)$ and $j$ is less than the multiplicity of $\lambda$.
\end{prpt}

\begin{proof} We assume that $q(z)$ has the form as in \eqref{eq:qz} and fix an $s$ with $1\le s\le r$.
 From \eqref{eq:lemma231}, we only need to prove that
\begin{align}\label{needtoprove}
A_{sj}(z_2)=a(\lambda_sz_2)_{j}^{\phi}b(z_2),\quad j=0,\cdots,k_s-1.\end{align}
Let us introduce the polynomials
\begin{eqnarray*}
q_s(z)=\prod_{1\le i\ne s\le r}((z-\lambda_s)^{k_s}-(\lambda_i-\lambda_s)^{k_s})^{k_i}.
\end{eqnarray*}
Note that $(z-\lambda_i)^{k_i}$ divides $q_s(z)$ and so there is a  $g_s(z)\in \C[z]$ such that
\begin{eqnarray}\label{eq:qsz}
q_s(z)(z-\lambda_s)^{k_s}=g_s(z)q(z).
\end{eqnarray}
Then from \eqref{eq:qsz} and \eqref{eq:lemadelta-1}, we have
\begin{equation}\begin{split}\label{eq:firstdec}
 &q_s(z_1/z_2)[a(z_1),b(z_2)]\\
 =\,&\sum _{j=0}^{k_s-1} q_s(z_1/z_2)A_{sj}(z_{2})\left(p(z_{2})\frac{\partial}{\partial z_{2}}\right)^{(j)}
\bar{p}(\lambda_s^{-1}z_1)\delta\left(\frac{\lambda_sz_{2}}{z_{1}}\right).
\end{split}\end{equation}

By replacing $z_1$ with $\lambda_s z_1$ in \eqref{eq:firstdec} and multiplying $(z_1/z_2-1)^{k_s}$ on both sides of \eqref{eq:firstdec} we have
\begin{eqnarray*}
(z_1/z_2-1)^{k_s}q_s(\lambda_s z_1/z_2)[a(\lambda_sz_1),b(z_2)]=0.
\end{eqnarray*}
Thus we have (see Definition \ref{de:yephi})
\begin{eqnarray*}
&&\quad(\phi(z,z_0)/z-1)^{k_s}q_s(\lambda_s \phi(z,z_0)/z)Y_\E^{\phi}(a(\lambda_sz),z_0)b(z)\\
&&=(z_1/z-1)^{k_s}q_s(\lambda_s z_1/z)a(\lambda_sz_1)b(z)\mid_{z_1=\phi(z,z_0)}.
\end{eqnarray*}
In view of Lemma \ref{lem:formal-jocobi}, this implies that
\begin{equation*}\begin{split}
&(z_{2}z)^{-1}\delta\left(\frac{z_{1}-z_{2}}{z_{2}z}\right)q_s(\lambda_s z_1/z_2)a(\lambda_sz_1)b(z_2)\\
&\quad\quad\quad\quad\quad\quad-(z_{2}z)^{-1}\delta\left(\frac{z_{2}-z_{1}}{-z_{2}z}\right)q_s(\lambda_s z_1/z_2)b(z_2)a(\lambda_sz_1)\\
&=z_{1}^{-1}\delta\left(\frac{z_{2}(1+z)}{z_{1}}\right)q_s(\lambda_s z_1/z_{2})Y_\E^{\phi}(a(\lambda_sz_{2}),f_{p(z_2)}(z_2,z))b(z_{2}).
\end{split}\end{equation*}

By taking $\Res _z z_2$  in the above identity we find
\begin{eqnarray*}
&&\quad q_s(\lambda_s z_1/z_2)[a(\lambda_s z_1),b(z_2)]\\
&&=\Res_z z_{1}^{-1}\delta\left(\frac{z_{2}(1+z)}{z_{1}}\right)z_2q_s(\lambda_s z_1/z_{2})Y_\E^{\phi}(a(\lambda_sz_{2}),f_{p(z_2)}(z_2,z))b(z_{2})\\
&&=\Res_{z_0} \frac{\partial}{\partial z_0 }\left(\frac{\phi(z_2,z_0)}{z_2}-1\right) z_{1}^{-1}\delta\left(\frac{\phi(z_2,z_0)}{z_{1}}\right)z_2q_s(\lambda_s z_1/z_{2})Y_\E^{\phi}( a(\lambda_sz_{2}),z_0)b(z_{2})\\
&&=\Res_{z_0} p(\phi(z_2,z_0)) z_{1}^{-1}\delta\left(\frac{\phi(z_2,z_0)}{z_{1}}\right)q_s(\lambda_s z_1/z_{2})Y_\E^{\phi}(a(\lambda_sz_{2}),z_0)b(z_{2})\\
&&=\Res_{z_0} p(z_{1})z_1^{-1}\delta\left(\frac{\phi(z_2,z_0)}{z_{1}}\right)q_s(\lambda_s z_1/z_{2})Y_\E^{\phi}(a(\lambda_sz_{2}),z_0)b(z_{2})\\
&&=\sum_{j\geq0}\left(p(z_{2})\frac{\partial}{\partial z_{2}}\right)^{(j)}\left(\bar{p}(z_1)\delta\left(\frac{z_{2}}{z_{1}}\right)\right)q_s(\lambda_s z_1/z_{2})
a(\lambda_s z_2)_{j}^{\phi}b(z_2),
\end{eqnarray*}
where in the second equality we have used the fact \eqref{hatphi}.
Combining this with \eqref{eq:firstdec} (replacing $z_1$ with $\lambda_s z_1$ therein) we obtain
\begin{align}\label{A_i-a_jb}
\sum_{j=0}^{k_s-1}
q_s\(\lambda_s \frac{z_1}{z_2}\)(A_{sj}(z_2)-a(\lambda_sz_2)_{j}^{\phi}b(z_2)) \left(p(z_{2})\frac{\partial}{\partial z_{2}}\right)^{(j)}\bar{p}(z_1)\delta\left(\frac{z_{2}}{z_{1}}\right)
=0.
\end{align}

Set
\[\alpha=\prod_{1\le i\ne s\le r}(-1)^{k_i}(\lambda_i-\lambda_s)^{k_sk_i}\in \C^\times.\]
Note that  there is an $h(z)\in \C[z]$ such that
\[q_s(\lambda_s z)=\prod_{1\le i\ne s\le r} ((\lambda_sz-\lambda_s)^{k_s}-(\lambda_i-\lambda_s)^{k_s})^{k_i}
=(z-1)^{k_s}h(z)+\alpha.\]
In view of this and \eqref{eq:lemadelta-1},  \eqref{A_i-a_jb} can be rewritten as
\begin{equation*}\begin{split}
\sum_{j=0}^{k_s-1}\alpha(A_{sj}(z_2)-a( \lambda_s z_2)_{j}^{\phi}b(z_2))\left(p(z_{2})\frac{\partial}{\partial z_{2}}\right)^{(j)}\bar{p}(z_1)\delta\left(\frac{z_{2}}{z_{1}}\right)=0.
\end{split}\end{equation*}
Note that $\alpha\ne 0$, this together with Lemma \ref{delta-2} proves \eqref{needtoprove}.
\end{proof}

\begin{remt}{\em
For $p(x)=1$ and $p(x)=x$, Proposition \ref{prop:quasilocalpair} was proved in \cite{Li5} and \cite{Li4} respectively.
}
\end{remt}

\section{$\phi$-coordinated modules of vertex algebras}
In this section we establish some conceptual results on   $\phi$-coordinated (quasi) modules
for vertex algebras.
Let $(V,Y,\mathbf{1})$ be a vertex algebra,
the following notion was introduced in \cite{Li2}.

\begin{dfnt}\label{de:phicoor}{\em
  A {\em $\phi$-coordinated quasi $V$-module} is a vector space $W$ equipped with a linear map
\begin{eqnarray*}
Y_{W}(\cdot,z):V&\rightarrow&\Hom(W,W((z)))\subset(\End W)[[z,z^{-1}]],\\
v&\mapsto &Y_{W}(v,z)=\sum_{n\in\Z}v_{n}z^{-n-1},
\end{eqnarray*}
satisfying the conditions that $Y_{W}(\mathbf{1},z)=1_{W}$ and that for $u,v\in V$, there exists
a polynomial $f(z_1,z_2)\in \C[z_1,z_2]$ such that
\begin{eqnarray*}
&&f(z_1,z_2)Y_{W}(u,z_{1})Y_{W}(v,z_{2})\in\Hom(W,W((z_{1},z_{2}))),\\
&&f(\phi(z_{2},z_{0}),z_2) Y_{W}(Y(u,z_{0})v,z_{2})=(f(z_1,z_2) Y_{W}(u,z_{1})Y_{W}(v,z_{2}))\mid_{z_{1}=\phi(z_{2},z_{0})}.
\end{eqnarray*}
A {\em $\phi$-coordinated $V$-module} is defined as above except that $f(z_1,z_2)$ is assumed to be a polynomial
of the form $(z_1-z_2)^k$ with $k\in \N$.
}
\end{dfnt}

\begin{remt}{\em
When $p(x)=1$, a $\phi$-coordinated $V$-module is just a $V$-module.}
\end{remt}

In what follows we collect some basic properties of $\phi$-coordinated quasi $V$-modules for later use.

\begin{lemt}\cite{Li2} Let $(W,Y_W)$ be a $\phi$-coordinated quasi $V$-module and let $u,v\in V$.
Assume that there exists a nonzero polynomial $q(z_1,z_2)\in \C[z_1,z_2]$ such that
\begin{align*}
q(z_1,z_2)Y_W(u,z_1)Y_W(v,z_2)\in \Hom(W,W((z_1,z_2))).
\end{align*}
Then
\begin{align}
\label{eq:qzlocal} q(z_1,z_2)[Y_W(u,z_1),Y_W(v,z_2)]&=0,\\
q(\phi(z_{2},z_{0}),z_2) Y_{W}(Y(u,z_{0})v,z_{2})&=(q(z_1,z_2) Y_{W}(u,z_{1})Y_{W}(v,z_{2}))\mid_{z_{1}=\phi(z_{2},z_{0})}.
\end{align}
\end{lemt}

We define the linear operator $\mathcal{D}$  on $V$  by $\mathcal{D}v=v_{-2}\vac$ for $v\in V$.

\begin{lemt}\label{lem:Dpz}\cite{Li2} Let $(W,Y_W)$ be a $\phi$-coordinated quasi $V$-module.
Then
\begin{align}
\label{dpzddz} Y_W(\mathcal{D}v,z)&=\left(p(z)\frac{d}{dz}\right)Y_W(v,z),\quad v\in V,\\
\label{ywhom} Y_W(Y(u,z_0)v,z)&=Y_\E^\phi(Y_W(u,z),z_0)Y_W(v,z),\quad u,v\in V.
\end{align}
\end{lemt}

The following is an analogue of Borcherd's commutator formula (cf. \cite{BLP}).
\begin{prpt}\label{prop-module-comut}
Let $(W,Y_{W})$ be a $\phi$-coordinated $V$-module. Then for $u,v\in V$,
\begin{equation*}\begin{split}
[Y_{W}(u,z_{1}),Y_{W}(v,z_{2})]=\sum_{j\geq0}Y_{W}(u_{j}v,z_{2})\left(p(z_2)\frac{\partial}{\partial z_{2}}\right)^{(j)}\bar{p}(z_1)\delta\left(\frac{z_{2}}{z_{1}}\right).
\end{split}\end{equation*}
Conversely, if $W$ is faithful and
\begin{equation*}\begin{split}
[Y_{W}(u,z_{1}),Y_{W}(v,z_{2})]=\sum_{j\geq0}Y_{W}(c_j,z_{2})\left(p(z_2)\frac{\partial}{\partial z_{2}}\right)^{(j)}\bar{p}(z_1)\delta\left(\frac{z_{2}}{z_{1}}\right),
\end{split}\end{equation*}
for some $c_j\in V$, then $c_{j}=u_{j}v$ for all $j\geq 0$.
\end{prpt}
\begin{proof} The second assertion follows from Lemma \ref{delta-2}.
For the first one, let $u,v\in V$. In view of \eqref{eq:qzlocal}, there is a nonnegative integer $k$ such that
\[(z_1-z_2)^k\,[Y_W(u,z_1),Y_W(v,z_2)]=0.\]
It then follows from Proposition \ref{prop:quasilocalpair} that
\begin{equation*}\begin{split}
[Y_W(u,z_1),Y_W(v,z_2)]&=\sum\limits_{j\ge 0} Y_W(u,z_2)^\phi_{j}Y_W(v,z_2)\left(p(z_{2})\frac{\partial}{\partial z_{2}}\right)^{(j)}\bar p(z_1)\delta\left(\frac{z_{2}}{z_{1}}\right).
\end{split}\end{equation*}
This together with \eqref{ywhom} proves the first assertion.
\end{proof}

Next we are going to give a conceptual  construction of $\phi$-coordinated quasi modules for vertex algebras.
The following is a $\phi$-generalization of \cite[Proposition 4.3]{Li5}.

\begin{prpt}\label{prop:n-product-local} Let $a(z),b(z),c(z)\in\E(W)$. Assume that
\begin{equation*}\begin{split}
f(z_1,z_2)[a(z_1),b(z_2)]=
g(z_1,z_2)[a(z_1),c(z_2)]=
h(z_1,z_2)[b(z_1),c(z_2)]=0
\end{split}\end{equation*}
for some nonzero polynomials $f(z_1,z_2),g(z_1,z_2),h(z_1,z_2)$. Then for any $n\in\Z$, there exists a nonnegative integer $k$ such that
\[g(z_1,z_2)^kh(z_1,z_2)[a(z_1)_{n}^\phi b(z_1),c(z_2)]=0.\]
\end{prpt}
\begin{proof}Fix an $n\in\Z$ and take a nonnegative integer $k$ such that
\[z_0^{k+n}f(\phi(z,z_0),z)^{-1}\in\C((z))[[z_0]].\]
Using this and the fact
\[\phi(\phi(z,z_0),-z_0)=\phi(z,0)=z=\phi(\phi(z,-z_0),z_0),\]
we obtain
\begin{eqnarray*}
&&\quad g(z_2,z_3)^kh(z_2,z_3)(a(z_2)_{n}^\phi b(z_2))c(z_3)\\
&&=\Res_{z_0}z_0^ng(z_2,z_3)^kh(z_2,z_3)\left(Y_\E^\phi(a(z_2),z_0)b(z_2)\right)c(z_3)\\
&&=\Res_{z_1}\Res_{z_0}z_0^ng(z_2,z_3)^kh(z_2,z_3)f(\phi(z_2,z_0),z_2)^{-1}\\
&&\quad z_1^{-1}\delta\left(\frac{\phi(z_2,z_0)}{z_1}\right)(f(z_1,z_2)a(z_1)b(z_2))c(z_3)\\
&&=\Res_{z_1}\Res_{z_0}z_0^ng(\phi(z_1,-z_0),z_3)^kh(z_2,z_3)f(\phi(z_2,z_0),z_2)^{-1}\\
&&\quad z_1^{-1}\delta\left(\frac{\phi(z_2,z_0)}{z_1}\right)(f(z_1,z_2)a(z_1)b(z_2))c(z_3)\\
&&=\Res_{z_1}\Res_{z_0}z_0^n\left(e^{-z_0p(z_1)\frac{\partial}{\partial z_1}}g(z_1,z_3)^k\right)h(z_2,z_3)f(\phi(z_2,z_0),z_2)^{-1}\\
&&\quad z_1^{-1}\delta\left(\frac{\phi(z_2,z_0)}{z_1}\right)(f(z_1,z_2)a(z_1)b(z_2))c(z_3)\\
&&=\Res_{z_1}\Res_{z_0}\sum_{i=0}^{k-1}z_0^{n+i}\left(\left(-p(z_1)\frac{\partial}{\partial z_1}\right)^{(i)}g(z_1,z_3)^k\right)h(z_2,z_3)f(\phi(z_2,z_0),z_2)^{-1}\\
&&\quad z_1^{-1}\delta\left(\frac{\phi(z_2,z_0)}{z_1}\right)(f(z_1,z_2)a(z_1)b(z_2))c(z_3)\\
&&=\Res_{z_1}\Res_{z_0}\sum_{i=0}^{k-1}z_0^{n+i}\left(\left(-p(z_1)\frac{\partial}{\partial z_1}\right)^{(i)}g(z_1,z_3)^k\right)h(z_2,z_3)c(z_3)\\
&&\quad f(\phi(z_2,z_0),z_2)^{-1}z_1^{-1}\delta\left(\frac{\phi(z_2,z_0)}{z_1}\right)(f(z_1,z_2)a(z_1)b(z_2))\\
&&=\Res_{z_1}\Res_{z_0}z_0^n\left(e^{-z_0p(z_1)\frac{\partial}{\partial z_1}}g(z_1,z_3)^k\right)h(z_2,z_3)c(z_3)f(\phi(z_2,z_0),z_2)^{-1}\\
&&\quad z_1^{-1}\delta\left(\frac{\phi(z_2,z_0)}{z_1}\right)(f(z_1,z_2)a(z_1)b(z_2))\\
&&=g(z_2,z_3)^kh(z_2,z_3)c(z_3)(a(z_2)_{n}^\phi b(z_2)).
\end{eqnarray*}
\end{proof}

A subset $S$ of $\E(W)$ is said to be {\em (resp.\,quasi) local} if any  pair in $S$ is (resp.\,quasi) local, and
a (resp.\,quasi) local subspace $U$ of $\E(W)$ is said to be {\em $Y_\E^\phi$-closed} if
\[a(z)^\phi_{n} b(z)\in U\quad\te{for all }a(z),b(z)\in U,\ n\in\Z.\]
In view of Proposition \ref{prop:n-product-local}, we can formulate the following notion.

\begin{dfnt}{\em Let $S$ be a quasi local subset of $\E(W)$.
We define $\<S\>_\phi$ to be the smallest $Y_\E^\phi$-closed quasi local subspace of $\E(W)$ which contains $S$ and $1_W$.}
\end{dfnt}

We need the following $\phi$-generalization of \cite[Proposition 4.8]{Li5} for later use.

\begin{prpt}\label{prop:quasilocaltolocal} For  $a(z),b(z)\in \E(W)$, if
\begin{align*}
f(z_1,z_2)\,[a(z_1),b(z_2)]=0
\end{align*}
for  some $f(z_1,z_2)\in \C[z_1,z_2]$,
then
\[(z_1-z_2)^k\,[Y_\E^\phi(a(z),z_1),Y_\E^\phi(b(z),z_2)]=0,\]
where $k$ is the order of the zero of $f(z_1,z_2)$ at $z_1=z_2$.
\end{prpt}
\begin{proof} Let $a(z),b(z)$, $f(z_1,z_2)$ and $k$ be as in the proposition.
Then it follows from \cite[Proposition 2.18]{JKLT} that
\begin{align}\label{fphiz1z2}
f(\phi(z,z_1),\phi(z,z_2))[Y_\E^\phi(a(z),z_1),Y_\E^\phi(b(z),z_2)]=0.
\end{align}

Write
\[f(z_1,z_2)=(z_1-z_2)^k f_0(z_1,z_2)\]
for some polynomial $f_0(z_1,z_2)$ with $f_0(z_1,z_1)\ne 0$.
Note that $\phi(z,0)=z$.
This implies that $f_0(\phi(z,0),\phi(z,0))$ is nonzero and hence $f_0(\phi(z,z_1),\phi(z,z_2))$ is
invertible in $\C((z))[[z_1,z_2]]$.
On the other hand, from the definition of $\phi(z,z_1)=e^{z_1p(z)\frac{d}{dz}}(z)$, we have
\begin{align*}
\phi(z,z_1)-\phi(z,z_2)=(z_1-z_2)g(z,z_1,z_2)
\end{align*}
for some  $g(z,z_1,z_2)\in \C((z))[[z_1,z_2]]$.
Furthermore, we have $g(z,0,0)=p(z)\ne 0$ and hence
$g(z,z_1,z_2)$ is also invertible in $\C((z))[[z_1,z_2]]$.
Thus, we have shown that
$f(\phi(z,z_1),\phi(z,z_2))$ is a product of $(z_1-z_2)^k$ and an invertible element in
$\C((z))[[z_1,z_2]]$.
This together with \eqref{fphiz1z2} proves the proposition.
\end{proof}

As the main result of this section we have:

\begin{thm}\label{thm-(quasi)localset-va} Let $S$ be a quasi local subset of $\E(W)$.
Then $(\langle S\rangle_{\phi},Y_\E^{\phi},1_W)$ is  a vertex algebra with  $S$ as a generating set and with $W$ as
 a faithful $\phi$-coordinated
  quasi module on which  $Y_W(a(z),z_0)=a(z_0)$ for $a(z)\in\langle S\rangle_{\phi}$.
  Furthermore, if $S$ is local, then $W$ is a $\phi$-coordinated
$\langle S\rangle_{\phi}$-module.
\end{thm}
\begin{proof}
It suffices to prove the case that $S$ is quasi local. Recall from \cite{Li2} that $\langle S\rangle_{\phi}$ is a
nonlocal vertex algebra with $W$ as a faithful $\phi$-coordinated quasi
$\langle S\rangle_{\phi}$-module.
On the other hand, it follows from Proposition \ref{prop:quasilocaltolocal} that
\[\{Y_{\E}^\phi(a(z),z_0)\mid a(z)\in\langle S\rangle_{\phi}\}\]
is a local subset of $\E(\langle S\rangle_{\phi})$.
This implies that $\<S\>_\phi$ is indeed a vertex algebra.
\end{proof}

\section{$\phi$-coordinated $V_{\CC}$-modules}

In this section, we first recall the notion of a conformal algebra $\CC$ \cite{K} (also known as
a vertex Lie algebra in \cite{DLM,P}) and then determine the $\phi$-coordinated modules for the universal enveloping
vertex algebra of $\CC$ in terms of restricted modules of a Lie algebra $\wh{\CC}_\phi$ arising from $\CC$ and $\phi$.

\begin{dfnt}
{\em A {\em conformal algebra} $\CC$ is a  $\C[\partial]$-module with a $\C$-bilinear product $a_{n}b$ for each $n\in\N$ such that the following axioms hold $(a,b,c\in \CC,\ m,n\in\N)$:
\begin{eqnarray*}
\te{(C0)}&\qquad&a_{n}b=0\quad\te{for }n\gg0,\\
\te{(C1)}&\qquad&(\partial a)_{n}b=-na_{n-1}b,\\
\te{(C2)}&\qquad&a_{n}b=-\sum_{j=0}^{\infty}(-1)^{j+n}\partial^{(j)}(b_{n+j}a),\\
\te{(C3)}&\qquad&a_{m}(b_{n}c)=\sum_{j=0}^{m}\binom{m}{j}(a_{j}b)_{m+n-j}c+b_{n}(a_{m}c).
\end{eqnarray*}
 A linear map  $\varphi:\CC\rightarrow\CC'$ of conformal algebras is called
 a \emph{homomorphism} if
 $$\varphi\partial=\partial\varphi\quad\te{and}\quad \varphi(a_{{n}}b)=\varphi(a)_{n}\varphi(b)\quad\te{ for }a,b\in\CC,\ n\in\N.$$
}
\end{dfnt}

Let $\CC$ be a conformal algebra. Recall that $\phi$ is an associate of $F(x,y)$, which is determined by a nonzero  $p(x)\in \C((x))$.
 We define a multiplication on the formal loop space $\CC\ot \C((x))$ by
 \begin{eqnarray}\label{lb-C_p0}
 [a\ot f(x),b\ot g(x)]=\sum_{i\geq0}(a_ib)\ot \left(\left(p(x)\frac{d}{d x}\right)^{(i)}f(x)\right)g(x)
\end{eqnarray}
for  $a,b\in\CC$ and $f(x),g(x)\in\C((x))$.
Form the quotient space
\begin{align*}
\wh{\CC}_\phi:=\CC\ot \C((x))/\mathrm{Im}\, (\partial \ot 1+1\ot p(x)\frac{d}{d x}).
\end{align*}
It was known from \cite{K} that the multiplication \eqref{lb-C_p0} defines a Lie algebra structure on the quotient space $\wh{\CC}_\phi$.
When $\phi=x+z$, i.e., $p(x)=1$, we also write
$\wh{\CC}=\wh{\CC}_\phi.$
For $a\in\CC$ and $f(x)\in\C((x))$, we set
\[\overline{a\ot f(x)}=\rho(a\ot f(x)),\] where $\rho:
\CC\ot \C((x))\rightarrow \wh{\CC}_\phi$ is the quotient map.

For $a\in\CC$, we set
\[a_{\phi}(z)=\sum_{n\in\Z}\overline{a\ot x^np(x)}z^{-n-1}\in \wh{\CC}_\phi[[z,z^{-1}]].\]

\begin{dfnt}\label{de:resccmod}{\em A $\widehat{\CC}_{\phi}$-module $W$ is called  {\em restricted} if  $a_{\phi}(z)\in \E(W)$ for any $a\in\CC$.
 }
\end{dfnt}

One notices from definition that
\begin{align}\label{eq:deraphiz}
\(p(z)\frac{d}{d z}\)a_\phi(z)=(\partial a)_\phi(z),\quad a\in \CC.
\end{align}
Furthermore, in view of \eqref{lb-C_p0}, we have for $a,b\in \CC$,
\begin{eqnarray}\label{lb-C_p}
[a_{\phi}(z),b_{\phi}(w)]=\sum_{i\geq0}(a_ib)_{\phi}(w) \left(p(w)\frac{\partial}{\partial w}\right)^{(i)}\bar p(z)\delta\left(\frac{w}{z}\right).
\end{eqnarray}
We have the following useful lemma.

\begin{lemt}\label{lem:resmod} Let $W$ be a vector space and $\CC$ be a conformal algebra. Assume that there is a linear map
 \[\pi:\CC\rightarrow \E(W),\quad a\mapsto a(z)=\sum_{n\in \Z}a_{[n]} z^{-n-1},\] such that for $a\in \CC$,
 \begin{align}\label{eq:derphiz}
  \(p(z)\frac{d}{d z}\)a(z)=(\partial a)(z),
  \end{align}
and  for $a,b\in \CC$,
\begin{eqnarray}\label{lb-C}
[a(z),b(w)]=\sum_{i\geq0}(a_ib)(w) \left(p(w)\frac{\partial}{\partial_w}\right)^{(i)}\bar p(z)\delta\left(\frac{w}{z}\right).
\end{eqnarray}
Then there is a restricted $\wh{\CC}_\phi$-module structure on $W$ with $a_\phi(z)=a(z)$ for $a\in \CC$.
\end{lemt}
\begin{proof} Let $a\in \CC$ and $f(x)\in \C((x))$.
Set
\begin{align*}
f_0(x)=p(x)^{-1}f(x)=\sum_{n\ge n_0} b_n x^n,
\end{align*}
where $n_0\in\Z$.
We define an action of $\overline{a\ot f(x)}$ on $W$ by
\begin{align}\label{modact}
\overline{a\ot f(x)}w=\sum_{n=n_0}^{n_{a,w}} b_n a_{[n]}w,\quad w\in W,
\end{align}
where $n_{a,w}$ is a sufficiently large integer such that $a_{[m]}w=0$ for all $m>n_{a,w}$.
It follows from \eqref{eq:derphiz} that this action is well-defined.
Furthermore, by comparing \eqref{lb-C_p}  with \eqref{lb-C}, one can check that \eqref{modact} defines a restricted $\wh{\CC}_\phi$-module structure on $W$
with $a_\phi(z)=a(z)$ for $a\in \CC$.
\end{proof}

Now we recall the notion of universal enveloping  vertex algebra of a conformal algebra $\CC$.
Set
$
\wh{\CC}_+=\rho(\CC\ot \C[[x]]),
$
which is a Lie subalgebra of $\wh{\CC}$.
Let $\C$ be the one-dimensional trivial  $\widehat{\CC}_{+}$-module. Then we have the induced $\widehat{\CC}$-module
$$V_{\CC}=\mathcal{U}(\widehat{\CC})\otimes_{\mathcal{U}(\widehat{\CC}_{+})}\C.$$
Set $\vac=1\otimes 1$. It is known that $\CC$ can be identified as a subspace of $V_{\CC}$ through the linear map
\[a\mapsto \overline{a\ot x^{-1}}\vac,\quad a\in \CC.\]

Note that each vertex algebra is naturally a conformal algebra by setting $\partial=\mathcal{D}$ and forgetting
the $n$-products with $n<0$.
The following result is well-known (cf. \cite{K,DLM,P,FB}).
\begin{prpt}\label{prop-CAhom-VAhom} There is a  vertex algebra structure on $V_{\CC}$ such that $\vac$ is the vacuum vector and that
\[Y(a,z)=a(z)=\sum_{n\in \Z} \overline{a\ot x^n}z^{-n-1}\] for $a\in \CC $.
Moreover, for any vertex algebra $V$, if $\varphi:\CC\rightarrow V$ is a homomorphism of conformal algebras, then $\varphi$ extends uniquely to a vertex algebra homomorphism from $V_{\CC}$ to $V$.
\end{prpt}

The following is the main result of this section.

\begin{thm}\label{thm:re-m}
 Let $\CC$ be a conformal algebra. Then any restricted $\widehat{\CC}_{\phi}$-module $W$ is
  naturally a $\phi$-coordinated $V_{\CC}$-module with $Y_{W}(a,z)=a_{\phi}(z)$ for $a\in\CC$.
   On the other hand, any $\phi$-coordinated $V_{\CC}$-module $W$ is naturally
 a restricted $\widehat{\CC}_{\phi}$-module with $a_{\phi}(z)=Y_{W}(a,z)$ for $ a\in\CC.$
\end{thm}
\begin{proof} Let $W$ be a restricted $\widehat{\CC}_{\phi}$-module. Then it follows from \eqref{lb-C_p}
that \[S=\{a_{\phi}(z)\mid a\in\CC\}\] is a local subset of $\mathcal{E}(W)$. In view of
Theorem \ref{thm-(quasi)localset-va},  there is a vertex algebra $\< S\>_{\phi}$ with $W$ as a faithful $\phi$-coordinated $\< S\>_{\phi}$-module
on which $Y_{W}(b(z),z_0)=b(z_0)$ for $b(z)\in \< S\>_{\phi}$.
This together with \eqref{lb-C_p} implies that for $a,b\in \CC$,
\begin{eqnarray*}
&&[Y_{W}(a_\phi(z_0),z),Y_{W}(b_\phi(z_0),w)]=[a_{\phi}(z),b_{\phi}(w)]\nonumber\\
&=&\sum_{i\geq0}(a_ib)_{\phi}(w) \left(p(w)\frac{\partial}{\partial w}\right)^{(i)}\bar p(z)\delta\left(\frac{w}{z}\right).
\end{eqnarray*}
As $W$ is faithful, it follows from Proposition \ref{prop-module-comut}  that
\begin{align}\label{eq:aphizhom}
a_\phi(z)_{i}^\phi b_\phi(z)=(a_ib)_\phi(z),\quad a,b\in \CC, i\in \N.
\end{align}

 We define a linear map $$\varphi:\ \CC\rightarrow \langle S\rangle_{\phi},\quad a\mapsto a_{\phi}(z).$$
It follows from  \eqref{dpzddz} and
\eqref{eq:deraphiz} that for $a\in \CC$,
\[\D\varphi(a)=\left(p(z)\frac{d}{d z}\right)a_{\phi}(z)=(\partial a)_{\phi}(z)=\varphi\partial(a).\]
 This together with \eqref{eq:aphizhom} implies that  $\varphi$ is a homomorphism of conformal algebras.
Then Proposition \ref{prop-CAhom-VAhom} says that $\varphi$ can be extended to a vertex algebra homomorphism from $V_{\CC}$ to $\langle S\rangle_{\phi}$.
Thus, $W$ becomes a $\phi$-coordinated $V_{\CC}$-module with
  \[Y_{W}(a,w)=Y_W(\varphi(a),w)=Y_W(a_\phi(z),w)=
  a_{\phi}(w)\quad \te{for}\ a\in \CC.\]

On the other hand, let $W$ be a $\phi$-coordinated $V_{\CC}$-module.
Recall from \eqref{dpzddz} that for $a\in \CC$,
\begin{align*}
Y_W(\partial a,z)=Y_W(\mathcal{D} a,z)=\left(p(z)\frac{d}{d z}\right) Y_W(a,z),
\end{align*}
and recall from  Proposition \ref{prop-module-comut} that for $a,b\in \CC$,
\begin{equation*}\begin{split}
[Y_{W}(a,z),Y_{W}(b,w)]=\sum_{i\geq0}Y_W(a_ib,w) \left(p(w)\frac{\partial}{\partial_w}\right)^{(i)}\bar p(z)\delta\left(\frac{w}{z}\right).
\end{split}\end{equation*}
Combining this with Lemma \ref{lem:resmod}, we see that  $W$ is  a restricted $\widehat{\CC}_{\phi}$-module with $a_{\phi}(z)=Y_{W}(a,z)$ for $a\in \CC$,
as desired.
\end{proof}

\begin{remt}
When  $p(x)\in \C[x,x^{-1}]$ is a Laurent polynomial, the Lie algebra  $\wh{\CC}_\phi$ can be defined in the following usual way:
$$\wh{\CC}_\phi:=\CC\ot\C[x,x^{-1}]/\te{Im}(\partial\ot 1+1\ot p(x)\frac{d}{dx}).$$
However, for the case $p(x)\in\C((x))$, the derivation $p(x)\frac{d}{dx}$ may be not well-defined on $\C[x,x^{-1}]$.
So we need to enlarge the usual loop algebra $\CC\ot \C[x,x^{-1}]$ to the formal loop algebra
$\CC\ot \C((x))$.
\end{remt}
\section{$(G,\chi_\phi)$-equivariant $\phi$-coordinated quasi modules}

Let $G$ be a group and let $\chi, \chi_\phi:G\rightarrow \C^\times$ be two linear characters on $G$.
In this section we generalize the previous results on $\phi$-coordinated  modules  for vertex algebras
to the $(G,\chi_\phi)$-equivariant $\phi$-coordinated quasi modules for $(G,\chi)$-vertex algebras.
We start with the notion of $(G,\chi)$-vertex algebra in \cite{JKLT} (cf. \cite{Li3, Li5}).

\begin{dfnt}
{\em  A {\em $(G,\chi)$-vertex algebra}
is a pair $(V,R)$ consisting of a vertex algebra $V$ and a group homomorphism $R:G\rightarrow \mathrm{GL}(V)$ such that $R_g(\vac)=\vac$ and
\[R_gY(v,z)R_g^{-1}=Y(R_g(v),\chi(g)z)\quad\te{for }g\in G,v\in V.\]
For $(G,\chi)$-vertex algebras $(V,R)$, $(V',R')$, a linear map $\psi:V\rightarrow V'$ is called a {\em $(G,\chi)$-vertex algebra homomorphism} if
$\psi$ is a vertex algebra homomorphism and $\psi\circ R_g=R'_g\circ\psi$ for $g\in G$.
}
\end{dfnt}

For a subgroup $\Gamma$ of $\C^\times$, we denote by $\C_\Gamma[x]$ the set of all monic polynomials in $\C[x]$ whose roots
contained in $\Gamma$.
The following notion was introduced in \cite{JKLT}.
\begin{dfnt}
{\em Let $(V,R)$ be a $(G,\chi)$-vertex algebra. A {\em $(G,\chi_\phi)$-equivariant $\phi$-coordinated quasi $V$-module} is a $\phi$-coordinated quasi $V$-module $(W,Y_W)$
such that
  \begin{align}\label{eq:Ywgchiphi}
Y_W(R_gv,z)=Y_W(v,\chi_\phi(g)^{-1}z)\quad\te{for }g\in G,v\in V,\end{align}
and that for $u,v\in V$, there exists $q(z)\in\C_{\chi_\phi(G)}[z]$ such that
\begin{eqnarray}\label{eq:qzquasilocalG}
q(z_1/z_2)Y_W(u,z_1)Y_W(v,z_2)\in\Hom(W,W((z_{1},z_{2}))).
\end{eqnarray}
}
\end{dfnt}

 Let $(V,R)$ be a $(G,\chi)$-vertex algebra and let $(W,Y_W)$ be
 a $(G,\chi_\phi)$-equivariant $\phi$-coordinated quasi $V$-module.
If  $\frac{d}{dz}Y_W(v,z)\neq0$ for some $v\in V$,
 then it was shown in \cite{JKLT} that
\begin{equation}\begin{split}\label{chi-chi_phi}
\phi(x,\chi(g)z)=\chi_\phi(g)\phi(\chi_\phi(g)^{-1}x,z)\quad \te{for }g\in G.
\end{split}\end{equation}
Note that this relation is equivalent to
\begin{equation}\begin{split}\label{chi-chi_phi:p(x)}
p(\chi_\phi(g)x)=\chi(g)^{-1}\chi_\phi(g)p(x)\quad \te{for }g\in G.
\end{split}\end{equation}
On the other hand, if  $\frac{d}{dz}Y_W(v,z)=0$ for any $v\in V$, then the quotient vertex algebra
$V/\ker Y_W$ is simply an associative algebra and $G$ acts trivially on $V/\ker Y_W$.
In view of this, we always assume that the characters $\chi$ and $\chi_\phi$ satisfy the relation \eqref{chi-chi_phi}.
Note that we have $\chi=\chi_\phi$ if $p(x)=1$ and $\chi=1$ if $p(x)=x$.

As  a generalization of Proposition \ref{prop-module-comut}, we have:

\begin{prpt}\label{prop-G-phi-m}Let $(V,R)$ be a $(G,\chi)$-vertex algebra and let $(W,Y_W)$ be a $(G,\chi_\phi)$-equivariant $\phi$-coordinated quasi $V$-module.
Let $\psi:\chi_\phi(G)\rightarrow G$ be a section of $\chi_\phi$.
 Then for $u,v\in V$,
\begin{equation}\label{eq:G-phi-m}\begin{split}
&[Y_W(u,z_1),Y_W(v,z_2)]\\
=\Res_{z_0}&\sum_{g\in \psi(\chi_\phi(G))}Y_W(Y(R_{g^{-1}}u,z_0)v,z_2)e^{z_0p(z_2)\frac{\partial}{\partial z_2}}\left(\chi(g)\bar p(z_1)\delta\left(\frac{\chi_\phi(g)z_2}{z_1}\right)\right).
\end{split}\end{equation}

\end{prpt}
\begin{proof} Let $u,v\in V$. From \eqref{eq:qzquasilocalG} and \eqref{eq:qzlocal},
it follows that there exists a
$q(z)\in\C_{\chi_\phi(G)}[z]$ such that
\[q(z_1/z_2)[Y_W(u,z_1),Y_W(v,z_2)]=0.\]
Then by applying Proposition \ref{prop:quasilocalpair} we have
\begin{eqnarray*}
&&[Y_W(u,z_1),Y_W(v,z_2)]\\
&=&\sum _{\lambda\in \chi_\phi(G)}\sum_{j\geq0} Y_W(u,\lambda z_2)_j^\phi Y_W(v,z_2)
\left(p(z_{2})\frac{\partial}{\partial z_{2}}\right)^{(j)}\bar{p}(\lambda^{-1}z_1)\delta\left(\frac{\lambda z_{2}}{z_{1}}\right).
\end{eqnarray*}
 Note that \eqref{chi-chi_phi:p(x)} implies that for $g\in G$,
 \begin{align*}
 \bar{p}(\chi_\phi(g)^{-1}z_1)=\chi_\phi(g)z_1^{-1}
 p(\chi_\phi(g^{-1})z_1)=\chi(g)z_1^{-1}p(z_1)=\chi(g)\bar p(z_1),
 \end{align*}
together with $\psi$ a section of $\chi_\phi$, we obtain
\begin{equation*}\begin{split}
&[Y_W(u,z_1),Y_W(v,z_2)]\\
=\Res_{z_0}&\sum_{g\in\psi(\chi_\phi(G))}
Y_\E^{\phi}\left(Y_W(u,\chi_\phi(g) z_2),z_0\right)Y_W(v,z_2)e^{z_0p(z_2)\frac{\partial}{\partial z_2}}\left(\chi(g)\bar p(z_1)\delta\left(\frac{\chi_\phi(g)z_2}{z_1}\right)\right).
\end{split}\end{equation*}
From \eqref{eq:Ywgchiphi} and \eqref{ywhom}, we have
\begin{equation*}\begin{split}
Y_\E^{\phi}\left(Y_W(u,\chi_\phi(g) z),z_0\right)Y_W(v,z)
=Y_\E^{\phi}\left(Y_W(R_{g^{-1}}u, z),z_0\right)Y_W(v,z)
=Y_W(Y(R_{g^{-1}}u,z_0)v,z)
\end{split}\end{equation*}
for  $g\in G$. Thus \eqref{eq:G-phi-m} follows.
\end{proof}

\begin{dfnt}
{\em Let $W$ be a vector space and let $\Gamma$ be a subgroup of $\C^\times$.  A subset $S$ of $\E(W)$ is said to {\em $\Gamma$-quasi local} if for any pair $a(z),b(z)$ in $S$,
 there exists $q(z)\in\C_\Gamma[z]$ such that
\[q(z_1/z_2)[a(z_1),b(z_2)]=0,\]
and is said to be {\em $\Gamma$-stable} if for any $a(z)\in S$ and $\lambda\in\Gamma$, $a(\lambda z)\in S$.
}
\end{dfnt}

For a $\Gamma$-quasi local and $\Gamma$-stable subset $S$ of $\E(W)$, $\<S\>_\phi$ is also $\Gamma$-quasi local (Theorem \ref{prop:n-product-local}) and $\Gamma$-stable (\cite{JKLT}).
Furthermore, from \cite[Theorem 3.11]{JKLT} and Theorem \ref{thm-(quasi)localset-va}, we immediately have
the following result.
\begin{thm}\label{thm:(G-chi-localset-va}
Let $S$ be a $\chi_\phi(G)$-stable and $\chi_\phi(G)$-quasi local subset of $\E(W)$.
Assume that the character $\chi_\phi$ is injective.
Then the vertex algebra $\langle S\rangle_\phi$ together with the group homomorphism
\[R:G\rightarrow \mathrm{GL}(\<S\>_\phi),\quad g\mapsto R_g\] is a $(G,\chi)$-vertex algebra and $W$ is a faithful
$(G,\chi_\phi)$-equivariant $\phi$-coordinated quasi $\langle S\rangle_\phi$-module with $Y_W(a(z),z_0)=a(z_0)$ for $a(z)\in \langle S\rangle_\phi$, where $R_g(a(z))= a(\chi_\phi(g)^{-1}z)$.
\end{thm}

The following notion was introduced in \cite{Li3}:

\begin{dfnt}{\em  A {\em $(G,\chi)$-conformal algebra} $(\CC,R)$ is a conformal
 algebra $\CC$
together with a group homomorphism
\begin{eqnarray*}
&&R:G\rightarrow \mathrm{GL}(\CC),\quad g\mapsto R_g,
\end{eqnarray*}
such that for any $a,b\in\CC$,
\begin{eqnarray*}
\te{(CR1)}&\qquad&\partial R_g=\chi(g)^{-1}R_g\partial,\\
\te{(CR2)}&\qquad&R_g(a_{m}b)=\chi(g)^{-(m+1)}(R_ga)_{m}(R_gb)\quad\te{for }m\in\N,\\
\te{(CR3)}&\qquad&(R_ga)_{m}b=0\quad \te{for all but finitely many }m\in\N, g\in G.
\end{eqnarray*}
For $(G,\chi)$-conformal algebras $(\CC,R)$, $(\CC',R')$, a linear map $\psi:\CC\rightarrow \CC'$ is called a {\em $(G,\chi)$-conformal algebra homomorphism} if
$\psi$ is a conformal algebra homomorphism and
$\psi\circ R_g=R'_g\circ \psi$ for $g\in G$.}
\end{dfnt}

Let $(\CC,R)$ be a $(G,\chi)$-conformal algebra.
For any $g\in G$, we define a linear automorphism $\hat R_g$ on $\CC\ot \C((x))$ by
\begin{align}
\hat R_g(a\ot f(x))=\chi(g)^{-1}R_g(a)\ot f(\chi_\phi(g)^{-1}x),\quad a\in \CC, f(x)\in \C((x)).
\end{align}

\begin{lemt}\label{lem:hatRg} For $g\in G$, one has
\begin{align*}
\hat R_g\circ (\partial\ot 1+1\ot p(x)\frac{d}{dx})
=\chi(g)(\partial\ot 1+1\ot p(x)\frac{d}{dx})\circ \hat R_g.
\end{align*}
\end{lemt}
\begin{proof}
For  $a\in\CC$, $f(x)\in\C((x))$, we have
\begin{equation*}\begin{split}
&\hat R_g\big(\partial a\ot f(x)+a\ot p(x)\frac{d f(x)}{d x}\big)\\
=&\chi(g)^{-1}\Big(R_g(\partial a)\ot f(\chi_\phi(g)^{-1}x)+R_ga\ot p(\chi_\phi(g)^{-1}x)\big(\frac{d f(x)}{d x}|_{x=\chi_\phi(g)^{-1}x}\big)\Big)\\
=&(\partial R_g a)\ot f(\chi_\phi(g)^{-1}x)+\chi_\phi(g)^{-1}R_ga\ot p(x)\big(\frac{d f(x)}{d x}|_{x=\chi_\phi(g)^{-1}x}\big)\\
=&\chi(g)(\partial\ot1+1\ot p(x)\frac{d}{dx})\hat R_g(a\otimes f(x)),
\end{split}\end{equation*}
where in the second equality we have used the relations (CR1) and \eqref{chi-chi_phi:p(x)}.
\end{proof}

In view of Lemma \ref{lem:hatRg}, $\hat R_g$ induces a linear automorphism on $\wh{\CC}_\phi$, which we still denote  as $\hat R_g$.
Then we have:

\begin{lemt}\label{lem:Gonhatcc} For  $g\in G$, $\hat R_g$ is a Lie automorphism of $\wh\CC_\phi$.
Furthermore, for any $u,v\in \wh\CC_\phi$,
$[\hat R_g(u),v]=0$ for all but finitely many $g\in G$.
\end{lemt}
\begin{proof} For $a,b\in \CC$ and $f(x),g(x)\in \C((x))$ we have
\begin{equation*}\begin{split}
&[\hat R_g(a\ot f(x)),\hat R_g(b\ot g(x))]\\
=&\chi(g)^{-2}[R_g(a)\ot f(\chi_\phi(g)^{-1}x),R_g(b)\ot g(\chi_\phi(g)^{-1}x)]\\
=&\chi(g)^{-2}\sum_{i\geq0}((R_ga)_i(R_gb))\ot \left(\left(p(x)\frac{d}{d x}\right)^{(i)}f(\chi_\phi(g)^{-1}x)\right)g(\chi_\phi(g)^{-1}x)\\
=&\chi(g)^{-1+i}\sum_{i\geq0}(R_g(a_ib))\ot \left(\left(p(x)\frac{d}{d x}\right)^{(i)}f(\chi_\phi(g)^{-1}x)\right)g(\chi_\phi(g)^{-1}x)\\
=&\hat R_g([a\ot f(x),b\ot g(x)]),
\end{split}\end{equation*}
where in the third  and last equality we have used the relations (CR2) and \eqref{chi-chi_phi:p(x)}, respectively.
This proves that $\hat R_g$ is a Lie automorphism and
the second assertion follows from (CR3).
\end{proof}

Following \cite{Li3}, we define a multiplication  on $\wh\CC_\phi$ by
\begin{align}
(u,v)\mapsto \sum_{g\in G} [\hat R_g(u),v],\quad u,v\in \wh\CC_\phi.
\end{align}
In view of Lemma \ref{lem:Gonhatcc}, it was known that under this new multiplication, the
quotient space
\begin{align*}
\wh\CC_\phi[G]:=\wh\CC_\phi/\te{span}\{\hat R_g(u)-u\mid g\in G, u\in \wh\CC_\phi\}
\end{align*}
is a Lie algebra \cite{Li3}.
For convenience, for $a\in \CC$, $f(x)\in \C((x))$, we still denote the image of $\overline{a\ot f(x)}$ in $\wh\CC_\phi[G]$ by $\overline{a\ot f(x)}$.

\begin{remt}{\em When $G$ is a finite group, $\wh\CC_\phi[G]$ is obviously isomorphic to the subalgebra
of $\wh\CC_\phi$ fixed by $\hat R_g$ for any $g\in G$.
}
\end{remt}

For $a\in\CC$, we set
\[a^G_{\phi}(z)=\sum_{n\in\Z}\overline{a\ot x^np(x)}z^{-n-1}\in \wh\CC_\phi[G][[z,z^{-1}]].\]
We have

\begin{dfnt}\label{de:resccgmod}{\em A $\widehat{\CC}_{\phi}[G]$-module $W$ is called  {\em restricted} if $a^G_{\phi}(z)\in \E(W)$ for any $a\in\CC$.
 }
\end{dfnt}

From \eqref{chi-chi_phi:p(x)}, we have that for $a\in \CC$,
 \begin{align}\label{eq:rgaphiz}
 (R_ga)^G_{\phi}(z)=a^G_{\phi}(\chi_\phi(g)^{-1}z).\end{align}
Furthermore, for $a,b\in \CC$, we have
\begin{equation}\begin{split}\label{lb-G2}
&[a^G_{\phi}(z),b^G_{\phi}(w)]=\sum_{g\in G}[(R_ga)^G_{\phi}(\chi_\phi(g)z),b^G_{\phi}(w)]\\
=&\sum_{g\in G}\sum_{i\geq 0}\left((R_{g^{-1}}a)_{i}b\right)^G_{\phi}(w)\left(p(w)\frac{\partial}{\partial w}\right)^{(i)}\left(\chi(g)\bar p(z)\delta\left(\frac{\chi_\phi(g)w}{z}\right)\right).
\end{split}\end{equation}

The same proof as that of Lemma \ref{lem:resmod}, we have
\begin{lemt}\label{lem:resgmod} Let $W$ be a vector space. Assume that there is a linear map
\[\pi:\CC\rightarrow \E(W),\quad a\mapsto a(z)=\sum_{n\in \Z}a_{[n]} z^{-n-1},\]
 such that for $a\in \CC$,
 \begin{align}\label{eq:derphiz-G}
  \(p(z)\frac{\partial}{\partial z}\)a(z)=(\partial a)(z),\quad
  (R_ga)(z)=a(\chi_\phi(g)^{-1}z),
  \end{align}
and  for $a,b\in \CC$,
\begin{eqnarray*}\label{lb-CG}
[a(z),b(w)]=
\sum_{g\in G}\sum_{i\geq 0}\left((R_{g^{-1}}a)_{i}b\right)(w)\left(p(w)\frac{\partial}{\partial w}\right)^{(i)}\left(\chi(g)\bar p(z)\delta\left(\frac{\chi_\phi(g)w}{z}\right)\right).
\end{eqnarray*}
Then there is a restricted $\wh{\CC}_\phi[G]$-module structure on $W$ with $a^G_\phi(z)=a(z)$ for $a\in \CC$.
\end{lemt}

 Let $(\CC,R)$ be a $(G,\chi)$-conformal  algebra and
set $H=\ker\chi_\phi$.  One notices from \eqref{chi-chi_phi:p(x)} that $\ker\chi_\phi\subset \ker \chi$, so
   $\chi$
  induces a linear character on $G/H$, denoted as $\bar\chi$.
Following \cite{Li3}, we have  a $(G/H,\bar\chi)$-conformal algebra $(\CC/H,R)$ induced from $(\CC,R)$.
Explicitly, as a vector space, \[\CC/H=\CC/\te{span}\{R_h a-a\mid h\in H,\ a\in\CC\}.\]
Define $\partial(aH)=(\partial a)H$ for $aH\in\CC/H$, and the $n$-products on $\CC/H$ by
\[(uH)_{n}(vH)=\sum_{h\in H}((R_hu)_{n}v)H\quad\te{ for } u,v\in\CC,\ n\in\N,\]
we also define the $G/H$-action  on $\CC/H$ by $R_{gH}(uH)=(R_gu)H$ for $g\in G, u\in \CC$.
Denote by $\bar\chi_\phi$ the character on $G/H$ induced by $\chi_\phi$.
Note that  the characters $\bar\chi,\bar\chi_\phi$
 still  satisfy the relation \eqref{chi-chi_phi:p(x)}.
 It is known that
the  $G/H$-action $R$  on $\CC/H$ can be extended uniquely to a $G/H$-action  on $V_{\CC/H}$, we still denote it by $R$,
then $(V_{\CC/H},R)$ becomes a $(G/H,\bar\chi)$-vertex algebra \cite{Li3}.
In particular we have a notion of
 $(G/H,\bar\chi_\phi)$-equivariant $\phi$-coordinated quasi module for the $(G/H,\bar\chi)$-vertex algebra $V_{\CC/H}$.
The following theorem is the main result of this section.

\begin{thm}\label{thm:G-re-m} Let $(\CC,R)$ be a $(G,\chi)$-conformal algebra and set $H=\ker\chi_\phi$.
  Then any restricted $\widehat{\CC}_{\phi}[G]$-module $W$ is a $(G/H,\bar\chi_\phi)$-equivariant
   $\phi$-coordinated quasi $V_{\CC/H}$-module
   with $Y_{W}(aH,z)=a^G_{\phi}(z)$ for $a\in \CC$. On the other hand, any $(G/H,\bar\chi_\phi)$-equivariant $\phi$-coordinated quasi $V_{\CC/H}$-module $W$ is  a restricted $\widehat{\CC}_{\phi}[G]$-module with $a^G_{\phi}(z)=Y_{W}(aH,z)$ for $a\in \CC$.
\end{thm}
\begin{proof}
Let $W$ be a restricted $\widehat{\CC}_{\phi}[G]$-module. Then $S=\{a^G_{\phi}(z)\mid a\in \CC\}$ is a $\chi_\phi(G)$-quasi local and $\chi_\phi(G)$-stable subset of $\E(W)$.
 Note that the character $\bar\chi_\phi:G/H\rightarrow \C^\times$ is injective and its image is $\chi_\phi(G)$.
In view of the Theorem \ref{thm:(G-chi-localset-va}, $\langle S\rangle_\phi$ is a $(G/H,\bar\chi)$-vertex algebra and $W$ is a faithful
$(G/H,\bar\chi_\phi)$-equivariant $\phi$-coordinated quasi $\langle S\rangle_\phi$-module with $Y_W(a(z),z_0)=a(z_0)$ for $a(z)\in \langle S\rangle_\phi$.
Note that for $h\in H$ and $a\in \CC$, from \eqref{eq:rgaphiz} we have that
\[(R_h a)^G_\phi(z)=a^G_\phi(\chi_\phi(h)^{-1}z)=a^G_\phi(z).\]
 Thus the linear map
 $$\varphi:\ \CC/H\rightarrow \langle S\rangle_{\phi},\quad aH\mapsto a^G_{\phi}(z)$$
 is well-defined. Take a section $\psi:\chi_\phi(G)\rightarrow G$ of $\chi_\phi$.
For $a,b\in \CC$, we have
\begin{align*}
&[a^G_{\phi}(z),b^G_{\phi}(w)]\\
=&\sum_{g\in G}\sum_{i\geq0}\left((R_{g^{-1}}a)_{i}b\right)^G_{\phi}(w)\left(p(w)\frac{\partial}{\partial w}\right)^{(i)}\left(\chi(g)\bar p(z)\delta\left(\frac{\chi_\phi(g)w}{z}\right)\right)\\
=&\sum_{g\in \psi(\chi_\phi(G))}\sum_{i\ge 0}
\sum_{h\in H}\left((R_{(gh)^{-1}}a)_{i}b\right)^G_{\phi}(w)\left(p(w)\frac{\partial}{\partial w}\right)^{(i)}\left(\chi(g)\bar p(z)\delta\left(\frac{\chi_\phi(g)w}{z}\right)\right).
\end{align*}
By comparing this with \eqref{eq:G-phi-m} we have
\begin{align*}
\sum_{h\in H}\left((R_{(gh)^{-1}}a)_{i}b\right)^G_{\phi}(w)
=(R_{g^{-1}}a)^G_\phi(w)_i^\phi b^G_\phi(w).
\end{align*}
This particular implies that
\begin{align*}
\varphi(aH)^\phi_i\varphi(bH)=\varphi((aH)_i(bH)),\quad a,b\in \CC, i\in \N.
\end{align*}
Furthermore, for any $a\in \CC$, we have
\begin{equation*}\begin{split}
&(\mathcal{D}\circ\varphi)(aH)=\mathcal{D}(a^G_\phi(z))=
\left(p(z)\frac{\partial}{\partial z}\right)a^G_{\phi}(z)=(\partial a)^G_{\phi}(z)=(\varphi\circ\partial)(aH).
\end{split}\end{equation*}
Thus, $\varphi$ is a conformal algebra homomorphism.
From the Proposition \ref{prop-CAhom-VAhom}, there is a vertex algebra homomorphism from $V_{\CC/H}$ to $\<S\>_\phi$
 which extends $\varphi$, we also denote it as $\varphi$. Note that for $a\in\CC$ and $g\in G$,
\begin{equation*}\begin{split}
&(R_{gH}\circ \varphi)(aH)=
R_{gH}a^G_\phi(z)=a^G_\phi(\bar{\chi}_\phi(gH)^{-1}z)\\
=&a^G_\phi(\chi_\phi(g)^{-1}z)=(R_ga)^G_\phi(z)
=\varphi((R_ga)H)=(\varphi\circ R_{gH})(aH).
\end{split}\end{equation*}
Then for $a,b\in \CC$ and $g\in G$, we have
 \begin{align*}
 &(R_{gH}\circ\varphi)(Y(aH,z)bH)
 =R_{gH}\big(Y_\E(\varphi(aH),z)\varphi(bH)\big)\\
 =&Y_\E(R_{gH}(\varphi(aH)),\bar{\chi}(gH)^{-1}z)R_{gH}(\varphi(bH))
 =Y_\E(\varphi(R_{gH}aH),\bar{\chi}(gH)^{-1}z)\varphi(R_{gH}bH)\\
 =&\varphi\big(Y(R_{gH}(aH),\bar{\chi}(gH)^{-1}z)R_{gH}(bH)\big)
 =(\varphi\circ R_{gH})(Y(aH,z)bH).
 \end{align*}
 This together with the fact that $\CC/H$ is a generating set of $V_{\CC/H}$
 proves that $R_{gH}\circ\varphi=\varphi\circ R_{gH}$ on $V_{\CC/H}$. Thus  $\varphi$ is a $(G/H,\bar\chi)$-vertex algebra homomorphism from $V_{\CC/H}$ to $\<S\>_\phi$.
Via this homomorphism, $W$ is naturally a $(G/H,\bar\chi_\phi)$-equivariant $\phi$-coordinated quasi $V_{\CC/H}$-module.

On the other hand, let $W$ be a
$(G/H,\bar\chi_\phi)$-equivariant $\phi$-coordinated quasi $V_{\CC/H}$-module.  From \eqref{dpzddz}  and \eqref{eq:rgaphiz}, for $a\in\CC$, we have
 \begin{align*}
&Y_W(\partial (aH),z)=Y_W(\mathcal{D} (aH),z)=\left(p(z)\frac{d}{d z}\right) Y_W(aH,z),\\
&(R_ga)^G_{\phi}(z)=a^G_{\phi}(\chi_\phi(g)^{-1}z).
\end{align*}
From Proposition \ref{prop-G-phi-m}, for $a,b\in\CC$, we have
\begin{equation*}\begin{split}
&[Y_W(aH,z),Y_W(bH,w)]\\
=&\sum_{g\in \psi(\chi_\phi(G))}\sum_{i\geq0}\sum_{h\in H}Y_W(((R_{hg^{-1}}a)_{i}b)H,w)\left(p(w)\frac{\partial}{\partial w}\right)^{(i)}\left(\chi(g)\bar p(z)\delta\left(\frac{\chi_\phi(g)w}{z}\right)\right)\\
=&\sum_{g\in G}\sum_{i\geq0}Y_W(((R_{g^{-1}}a)_{i}b)H,w)\left(p(w)\frac{\partial}{\partial w}\right)^{(i)}\left(\chi(g)\bar p(z)\delta\left(\frac{\chi_\phi(g)w}{z}\right)\right).
\end{split}\end{equation*}
Thus, from Lemma \ref{lem:resgmod},  we obtain that $W$ is a restricted $\widehat{\CC}_{\phi}[G]$-module with $a^G_{\phi}(z)=Y_{W}(aH,z)$ for $a\in \CC$.
\end{proof}
\section{Applications}
In this section, as an application, a class of $(G,\chi_\phi)$-equivariant $\phi$-coordinated quasi modules for affine vertex algebras is determined, and  the $\phi$-coordinated modules for
Virasoro vertex algebra and the vertex algebras associated to Novikov algebras are determined.

Firstly, let $\g$ be a (possibly infinite dimensional) Lie algebra, which is
  equipped with an invariant symmetric bilinear form $(\cdot,\cdot)$.
  Then we have the (formal) untwisted affine Lie algebra
  \begin{align*}
  \wh\g:=\g\ot \C((x))\oplus \C \mathbf{c},
  \end{align*}
  where $\mathbf{c}$ is central and for $a,b\in \g$, $f(x),g(x)\in \C((x))$,
  \begin{align*}
  [a\ot f(x), b\ot g(x)]=[a,b]\ot f(x)g(x)+(a,b)\mathrm{Res}_x((\frac{d}{dx}f(x))g(x))\mathbf{c}.
  \end{align*}
   On the other hand, we have the conformal algebra (\cite{K})
 \[\CC_\g=(\C[\partial]\ot\g)\op \C \mathrm{c},\]
 where  for $a,b\in \g$ and $n\in \N$, $\partial(\partial^n\ot a)=\partial^{n+1}\ot a,$ $\partial(\mathrm{c})=0,$ and
 \begin{align*}
a_nb=\delta_{n,0}[a,b]+\delta_{n,1}(a,b)\mathrm{c},\quad a_n\mathrm{c}=0=\mathrm{c}_n\mathrm{c}.
 \end{align*}
We have

  \begin{lemt}\label{lem:affiso} The assignment  $(a\in \g,f(x)\in \C((x)))$
 \begin{align*}
 \overline{a\ot f(x)}\mapsto a\ot f(x),\quad \overline{\mathrm{c}\ot x^{-1}p(x)}\mapsto \mathbf{c}
 \end{align*}
determines a
 Lie algebra isomorphism from $\wh{\CC_\g}_\phi$ to $\wh\g$.
 \end{lemt}

\begin{proof}It is obvious that this assignment gives rise to a linear isomorphism, say $\varphi$, from $\wh{\CC_\g}_\phi$ to $\wh\g$. For any $a,b\in\g$, $f(x),g(x)\in\C((x))$, we have
\begin{equation*}\begin{split}
&\varphi([\overline{a\ot f(x)},\overline{b\ot g(x)}])\\
=&\varphi(\overline{[a,b]\ot f(x)g(x)}+(a,b)\overline{\mathrm{c}\ot p(x)(\frac{d}{dx}f(x))g(x)})\\
=&[a,b]\ot f(x)g(x)+(a,b)\mathrm{Res}_x((\frac{d}{dx}f(x))g(x))\mathbf{c}\\
=&[\varphi(\overline{a\ot f(x)}),\varphi(\overline{b\ot g(x)})],
\end{split}\end{equation*}
noticing that \[\overline{\mathrm{c}\ot p(x)(\frac{d}{dx}f(x))g(x)}=\overline{\mathrm{c}\ot x^{-1}p(x)}\mathrm{Res}_x((\frac{d}{dx}f(x))g(x)).\] Thus $\varphi$ is a Lie algebra isomorphism.
\end{proof}

  For any $\ell \in \C$,
we denote by
  \[V_{\wh\g}(\ell,0)=\U(\wh\g)\ot_{\U(\g\ot \C[[x]]\oplus \C\mathbf{c})}\C\]
  the affine vertex algebra of $\g$ with the central charge $\ell$, where $\g\ot \C[[x]]$ acts
  trivially on $\C$ and $\mathbf{c}$ acts as the scalar $\ell$ on $\C$.
  Note that $V_{\wh\g}(\ell,0)$ is naturally isomorphic to
  the quotient vertex algebra of $V_{\CC_\g}$ modulo the ideal generated by $\mathrm{c}-\ell$.
  We say that a $\wh\g$-module $W$ is restricted if for any $a\in \g, w\in W$, $(a\ot x^n).w=0$ for $n\gg 0$,
  and is of level $\ell$ if $\mathbf{c}$ acts as the scaler $\ell$ on $W$.
From Lemma \ref{lem:affiso} and
Theorem \ref{thm:re-m}, we have the following characterization of $\phi$-coordinated  $V_{\wh\g}(\ell,0)$-modules.

  \begin{prpt}\label{prop:affvaphimod} $\phi$-coordinated modules for the affine vertex algebra $V_{\wh\g}(\ell,0)$ are exactly restricted $\wh\g$-modules of level $\ell$.
 \end{prpt}

 More generally, assume that $G$ is a group acting on $\g$ as an automorphism group preserving the bilinear form
$(\cdot,\cdot)$ such that
for $a,b\in \g$,
\begin{align}\label{conconvar}
[ga,b]=0\quad \te{and}\quad (ga,b)=0\quad \te{for all but finitely many}\ g\in G.
\end{align}
Let $\chi,\chi_\phi:G\rightarrow \C^\times$ be any linear characters satisfy the relation \eqref{chi-chi_phi}.
We lift the $G$-action from $\g$ to  $\wh\g$ by
\begin{align}
g(a\ot f(x)+\beta \mathbf{c})=(ga\ot f(\chi_\phi(g)^{-1}x))+\beta \mathbf{c}
\end{align}
for $g\in G$, $a\in \g$, $f(x)\in \C((x))$, $\beta\in \C$.
From \cite{Li3}, the {\em $(G,\chi_\phi)$-covariant algebra}  of $\wh\g$
is the Lie algebra $\wh\g[G]$, where
\begin{align}\label{kgamma}
\wh\g[G]=\wh\g/{\rm span}\{gu-u\mid g\in G, u\in \wh\g\}
\end{align}
as a vector space and the Lie bracket  is given by
\begin{align*}
[\overline{u},\overline{v}]=\sum_{g\in G} \overline{[gu,v]}\quad \te{for}\ u,v\in \wh\g,
\end{align*}
where for $u\in \wh\g$, $\overline{u}$ stands for the image of $u$ in $\wh\g[G]$
under the natural quotient map
$\wh\g\rightarrow \wh\g[G]$.
  We define a $G$-action  on $\CC_\g$ by
\[R_g(\partial^n\ot a)=\chi(g)^{n+1}\partial^n\ot ga\quad\te{and}\quad R_g(\mathrm{c})=\mathrm{c}\quad\te{for }g\in G, a\in\g, n\in\N.\]
Then   $(\CC_\g,R)$ is a   $(G,\chi)$-conformal algebra and so $(V_{\CC_\g},R)$ is a $(G,\chi)$-vertex algebra.
A similar proof as that in Lemma \ref{lem:affiso}, we have
\begin{lemt}\label{lem:twisted-af}  The assignment $(a\in \g,f(x)\in \C((x)))$
 \begin{align*}
 \overline{a\ot f(x)}\mapsto \overline{a\ot f(x)},\quad \overline{\mathrm{c}\ot x^{-1}p(x)}\mapsto \overline{\mathbf{c}}
 \end{align*}
determines a
 Lie algebra isomorphism from $\widehat{\CC_\g}_\phi[G]$ to $\wh\g[G]$.
\end{lemt}
 As $R_g(\mathrm{c})=\mathrm{c}$ for $g\in G$, $V_{\wh\g}(\ell,0)$ becomes a $(G,\chi)$-vertex algebra with the $G$-action $R$ transferred from
  that of $V_{\CC_\g}$.
As before, we can define the notion ``restricted $\wh\g[G]$-module of level $\ell$".
Then from Lemma \ref{lem:twisted-af} and
Theorem \ref{thm:G-re-m}, we immediately have the following generalization of Proposition \ref{prop:affvaphimod}.

\begin{prpt}\label{wh-g-G} Let $\chi,\chi_\phi$ be linear characters of $G$ with $\chi_\phi$ is injective and  the relation \eqref{chi-chi_phi} hold.
Then $(G,\chi_\phi)$-equivariant $\phi$-coordinated quasi modules for the affine $(G,\chi)$-vertex algebra $V_{\wh\g}(\ell,0)$
are exactly restricted $\wh\g[G]$-modules of level $\ell$.
\end{prpt}

\begin{remt} We illustrate two examples for the Proposition \ref{wh-g-G}(cf.\cite{Li3}).
First, let $\sigma$ be an automorphism of $\g$ which has order $M$ and preserves the form $(\cdot,\cdot)$, and
let $\omega$ be a primitive $M$-th root of unity.
Set $G=\<\sigma\>$ and take $\chi_\phi$ to be the character of $G$  determined by $\chi_\phi(\sigma)=\omega$.
Then $\wh\g[G]$ is isomorphic to the $\sigma$-twisted affine Lie algebra
 \[\wh{\g}[\sigma]=\{u\in \wh\g\mid \hat{\sigma}(u)=u\},\]
 where $\hat{\sigma}$ is the automorphism of $\wh{\g}$ defined by
 $\hat{\sigma}(\mathbf{c})=\mathbf{c}$ and
 $\hat{\sigma}(a\ot f(x))=\sigma(a)\ot f(\omega^{-1}x),$ $a\in \g, f(x)\in \C((x))$.

Second, let $\g=\fgl_\infty$ be the Lie algebra of  $\Z\times\Z$-matrices which have only finitely many nonzero entries
 and set $(x,y)=Tr(xy)$. For any positive integer $N$, define an automorphism $\sigma_N$ on $\fgl_\infty$ by
$\sigma_N(E_{i,j})=E_{i+N,j+N}$ for $i,j\in \Z$, where $E_{i,j}$ stands for the usual unit matrix.
 Set $G=\<\sigma_N\>$ and take $\chi_\phi$ to be the character of $G$ determined by $\chi_\phi(\sigma_N)=q^{-1}$ with $q$ is not a root of unit.
As $G\cong \Z$ is an infinite group, the relation \eqref{chi-chi_phi} forces that $p(x)$ must be a homogenous polynomial in $\C[x,x^{-1}]$ (\cite{JKLT}).
Thus we may define $\wh\g$ to be the usual affine Lie algebra $\g\ot \C[x,x^{-1}]\oplus \C\mathbf{c}$.
Then $\wh\g[G]$ is isomorphic to a one dimensional central extension of the $N\times N$-matrix algebra $\fgl_N(\C_q)$ over a $2$-dimensional
quantum torus $\C_q$,
where $\C_q=\C[x^{\pm{1}},y^{\pm1}]$ as vector space with the multiplication given by $(x^my^n)(x^sy^t)=q^{ns}x^{m+s}y^{n+t}$ (cf.\cite{Li3,CLTW}).
\end{remt}

Next, we consider the $\phi$-coordinated modules for the Virasoro vertex algebra.
Let
\begin{align*}
\mathcal{V}=\C((x))\frac{d}{dx}\oplus \C\mathbf{c}
\end{align*}
be (formal) Virasoro Lie algebra, where $\mathbf{c}$ is central and for $f(x),g(x)\in \C((x))$,
\begin{align*}
[f(x)\frac{d}{dx},g(x)\frac{d}{dx}]=&f(x)(\frac{d}{dx}g(x))\frac{d}{dx}
-g(x)(\frac{d}{dx}f(x))\frac{d}{dx}\\
&+\frac{\mathbf{c}}{2} \mathrm{Res}_x ((\frac{d}{dx})^{(3)}f(x))g(x).
\end{align*}
On the other hand, we have the  {\em Virasoro conformal algebra} (\cite{K})
\[\CC_\mathcal{V}
=(\C[\partial]\ot \C L)\op\C \mathrm{c},\]
where for $n\in \N$, $\partial(\partial^n\ot L)=\partial^{n+1}\ot L,$ $\partial \mathrm{c}=0$ and
\begin{align*}
L_nL=\delta_{n,0}\partial L+\delta_{n,1}2L+\delta_{n,3}\frac{\mathrm{c}}{2},\quad
L_n\mathrm{c}=0=\mathrm{c}_n\mathrm{c}.
\end{align*}
We identify $L$ as $1\ot L\in\CC_\mathcal{V}$.
It is straightforward to see that the Lie algebra $\wh{\CC_\mathcal{V}}_\phi$ is spanned by
the following elements
\begin{align*}
\overline{\mathrm{c}\ot x^{-1}p(x)},\quad \overline{L\ot f(x)},\ f(x)\in \C((x)).
\end{align*}
Furthermore, we have

  \begin{lemt}\label{lem:viriso} The assignment
 \begin{align*}
 \overline{L\ot f(x)}\mapsto p(x)f(x)\frac{d}{dx}+\alpha_\phi(f(x))\mathbf{c},\quad \overline{\mathrm{c}\ot x^{-1}p(x)}\mapsto \mathbf{c}
 \end{align*}
determines a
 Lie algebra isomorphism from $\wh{\CC_\mathcal{V}}_\phi$ to $\mathcal{V}$, where $\alpha_\phi:\C((x))\rightarrow\C$
 is a linear map defined by
 $$f(x)\mapsto\frac{1}{24}\Res_xf(x)\big(2(\frac{d}{dx})^2p(x)-p(x)^{-1}(\frac{d}{dx}p(x))^2\big).$$
 \end{lemt}

\begin{proof}Obviously, this assignment gives rise to a linear isomorphism $\varphi: \wh{\CC_\mathcal{V}}_\phi\rightarrow\mathcal{V}$. Now we prove that $\varphi$ is a Lie algebra homomorphism. For any $f(x),g(x)\in\C((x))$, we have
\begin{eqnarray*}
&&\quad\varphi\Big([\overline{L\ot f(x)},\overline{L\ot g(x)}]\Big)\\
&&=\varphi\Big(-\overline{L\ot p(x)\frac{d}{dx}(f(x)g(x))}+2\overline{L\ot p(x)(\frac{d}{dx}f(x))g(x)}\\
&&\quad+\frac{1}{2}\overline{\mathrm{c}\ot((p(x)\frac{d}{d x})^{(3)}f(x))g(x)}\Big)\\
&&=\varphi\Big(\overline{L\ot p(x)g(x)\frac{d}{dx}f(x)}-\overline{L\ot p(x)f(x)\frac{d}{dx}g(x)}\\
&&\quad+\frac{1}{12}\overline{\mathrm{c}\ot x^{-1}p(x)}\Res_x\Big((\frac{d}{dx}p(x))^2\frac{d}{dx}f(x)+3p(x)(\frac{d}{dx}p(x))(\frac{d}{dx})^2f(x)\\
&&\quad+p(x)((\frac{d}{dx})^2p(x))\frac{d}{dx}f(x)+p(x)^2(\frac{d}{dx})^3f(x)\Big)g(x)\Big)\\
&&=p(x)^2f(x)(\frac{d}{dx}g(x))\frac{d}{dx}-p(x)^2g(x)(\frac{d}{dx}f(x))\frac{d}{dx}\\
&&\quad+\frac{1}{12} \mathbf{c}\Res_x\Big(f(x)(\frac{d}{dx})^3p(x)+3(\frac{d}{dx}p(x))(\frac{d}{dx})^2f(x)\\
&&\quad+3((\frac{d}{dx})^2p(x))\frac{d}{dx}f(x)+p(x)(\frac{d}{dx})^3f(x)\Big)p(x)g(x)\\
&&=p(x)^2f(x)(\frac{d}{dx}g(x))\frac{d}{dx}-p(x)^2g(x)(\frac{d}{dx}f(x))\frac{d}{dx}+\frac{\mathbf{c}}{2} \mathrm{Res}_x (\frac{d}{dx})^{(3)}(p(x)f(x))p(x)g(x)\\
&&=[ p(x)f(x)\frac{d}{dx}, p(x)g(x)\frac{d}{dx}]=[\varphi(\overline{L\ot f(x)}),\varphi(\overline{L\ot g(x)})],
\end{eqnarray*}
noticing that \[\overline{\mathrm{c}\ot p(x)h(x)}=\overline{\mathrm{c}\ot x^{-1}p(x)}\mathrm{Res}_xh(x)\quad\te{for any }h(x)\in\C((x)).\] Thus $\varphi$ is a Lie algebra isomorphism.
\end{proof}

 For any $\ell \in \C$,
  we denote by
  \[V_{\mathcal{V}}(\ell,0)=\U(\mathcal{V})\ot_{\U(\C[[x]]\frac{d}{dx}\oplus \C\mathbf{c})}\C\]
  the Virasoro vertex algebra with central charge $\ell$, where $\C[[x]]\frac{d}{dx}$ acts
  trivially on $\C$ and $\mathbf{c}$ acts as the scalar $\ell$ on $\C$.
Then $V_{\mathcal{V}}(\ell,0)$ is naturally isomorphic to
  the quotient vertex algebra of $V_{\CC_\mathcal{V}}$ modulo the ideal generated by $\mathrm{c}-\ell$.
  We say that a $\mathcal{V}$-module $W$ is restricted if for any $w\in W$, $(x^n\frac{d}{dx}).w=0$ for $n\gg 0$,
  and is of level $\ell$ if $\mathbf{c}$ acts as the scaler $\ell$ on $W$.
From Lemma \ref{lem:viriso} and
Theorem \ref{thm:re-m}, we immediately have the following  characterization
of $\phi$-coordinated $V_{\mathcal{V}}(\ell,0)$-modules.

  \begin{prpt}\label{prop:virvaphimod} $\phi$-coordinated modules for the Viraroso vertex algebra $V_{\mathcal{V}}(\ell,0)$ are exactly restricted $\mathcal{V}$-modules of level $\ell$.
 \end{prpt}
 \begin{remt} When $p(x)=x$, Propositions \ref{prop:affvaphimod} and \ref{prop:virvaphimod} were proved in \cite{CLT} by a different method.
 \end{remt}

Finally we consider the $\phi$-coordinated modules of the vertex algebras associated to Novikov algebras.
Let $\CA$ be a Novikov algebra, in the sense that $\CA$ is a non-associative algebra satisfying \[(ab)c-a(bc)=(ba)c-b(ac),\quad (ab)c=(ac)b\quad \te{for }a,b,c\in\CA,\]
and let $\langle\cdot,\cdot\rangle$ be a symmetric bilinear form of $\CA$ satisfying
\begin{equation}\label{Novikov-form}\begin{split}
\langle ab,c\rangle=\langle a,bc\rangle,\quad \langle ab,c\rangle=\langle ba,c\rangle\quad \te{for } a,b,c\in\CA.
\end{split}\end{equation}
Similar to the algebra $\wh\g$, we also have an ``untwisted affine Lie algebra"
\begin{align*}
\wh{\CA}=\CA\ot \C((x))\oplus \C\mathbf{c}
\end{align*}
of $\CA$ (\cite{BLP}), where $\mathbf{c}$ is central and for $a,b\in\CA, f(x),g(x)\in\C((x))$,
\begin{equation*}\begin{split}
[a\ot f(x),b\ot g(x)]=&ab\ot\left(\frac{d}{d x}f(x)\right)g(x)-ba\ot\left(\frac{d}{d x}g(x)\right)f(x)\\
&+\frac{\mathbf{c}}{2}\<a,b\> \mathrm{Res}_x\ \big((\frac{d}{dx})^{(3)}f(x)\big)g(x).
\end{split}\end{equation*}
On the other hand, following \cite{P}, we have the conformal algebra
\[\CC_\CA=\C[\partial]\ot \CA\op\C \mathrm{c},\]
where for $a,b\in\CA$, $n\in \N$, $\partial(\partial^n\ot a)=\partial^{n+1}\ot a,$ $\partial \mathrm{c}=0$ and
\begin{align*}
a_nb=\delta_{n,0}\partial (ba)+\delta_{n,1}(ab+ba)+\delta_{n,3}\<a,b\>\frac{\mathrm{c}}{2},\quad
a_n\mathrm{c}=0=\mathrm{c}_n\mathrm{c}.
\end{align*}
Note that there is a linear isomorphism from $\wh{\CC_\CA}_\phi\rightarrow \wh{\CA}$ such that
$\overline{a\ot f(x)}\rightarrow a\ot f(x)$ for $a\in \CA, f(x)\in \C((x))$  and that $\overline{\mathrm{c}\ot x^{-1}p(x)}\mapsto \mathbf{c}$.
Via this isomorphism, we obtain a new Lie algebra structure on $\wh{\CA}$ such that
\begin{equation*}\begin{split}
[a\ot f(x),b\ot g(x)]=&ab\ot\left(p(x)\frac{d}{d x}f(x)\right)g(x)-ba\ot\left(p(x)\frac{d}{d x}g(x)\right)f(x)\\
&+\frac{\mathbf{c}}{2}\<a,b\> \mathrm{Res}_x\ p(x)^{-1}\big((p(x)\frac{d}{dx})^{(3)}f(x)\big)g(x)
\end{split}\end{equation*}
for $a,b\in\CA, f(x),g(x)\in\C((x))$ and that $\mathbf{c}$ is central.
We denote this Lie algebra by $\wh{\CA}_\phi$.

For any $\ell\in \C$, just as $V_{\wh\g}(\ell,0)$, we have an affine type vertex algebra $V_{\wh\CA}(\ell,0)$ (\cite{BLP})
which is isomorphic to the quotient vertex algebra $V_{\CC_\CA}/\<\mathrm{c}-\ell\>$.
We say that a $\wh{\CA}_\phi$-module  $W$ is restricted if for any $a\in \CA, w\in W$, $(a\ot x^n).w=0$ for $n\gg 0$,
  and is of level $\ell$ if $\mathbf{c}$ acts as the scaler $\ell$ on $W$.
In view of Theorem \ref{thm:re-m}, we have the following slight generalization of a result in \cite{BLP}:
  \begin{prpt} $\phi$-coordinated $V_{\wh\CA}(\ell,0)$-modules are exactly
  restricted $\wh\CA_\phi$-modules of level $\ell$.
\end{prpt}

  \begin{remt}
Equip $\CA=\C[y,y^{-1}]$ with a bilinear operator
\[\circ: \CA\times \CA\rightarrow\CA,\quad (y^i,y^j)\mapsto jy^{i+j}.\]
Then $(\CA,\circ)$ is a Novikov algebra.
Unlike the usual affine Lie algebra case (cf. Lemma \ref{lem:affiso}),  it follows from \cite{DZ} that the Lie algebra $\wh\CA$
is not isomorphic to $\wh\CA_\phi$ with $\phi=xe^z$.
  \end{remt}


\begin{thebibliography}{BLP}






\bibitem[BLP]{BLP}
C. Bai, H.-S. Li, Y. Pei, $\phi_{\epsilon}$-Coordinated modules for vertex algebras,  {\em J. Algebra}  {\bf 426} (2015), 211-242.



\bibitem[CLT]{CLT}F. Chen, H.-S. Li, S. Tan, Toroidal extended affine Lie algebras and vertex algebras, arXiv:2102.10968.

\bibitem[CLTW]{CLTW}F. Chen, H.-S. Li, S. Tan, Q. Wang, Extended affine Lie algebras,  vertex algebras, and reductive groups,
arXiv:2004.02821.

\bibitem[DLM]{DLM}C. Dong, H.-S. Li, G. Mason, Vertex Lie algebras, vertex Poisson algebras and vertex algebras, in: Recent Developments in
Infinite-Dimensional Lie Algebras and Conformal Field Theory, Charlottesville, VA, 2000, in: Contemp. Math., Vol. {\bf 297}, Amer.
Math. Soc., Providence, RI, 2002, pp. 69-96.

\bibitem[DZ]{DZ}D. Dokovic, K. Zhao, Derivations, isomorphisms and second cohomology of generalized Block algebras, {\em Algebra Colloq.} {\bf 3} (1996), 245-272.

\bibitem[FB]{FB} E. Frenkel, D. Ben-Zvi, {\em Vertex Algebras and Algebraic Curves,} Mathematical Surveys and Monographs, Vol. {\bf 88},  Amer. Math. Soc., Providence, 2004.

\bibitem[FLM]{FLM}
I. B. Frenkel, J. Lepowsky, A. Meurman, {\em Vertex Operator
Algebras and the Monster,} Pure and Applied Math., Vol. {\bf 134},
Academic Press, Boston, 1988.

\bibitem[FZ]{FZ}
I. B. Frenkel, Y.-C. Zhu,  Vertex operator algebras associated to
representations of affine and Virasoro algebras, {\em Duke Math. J.}
 {\bf 66} (1992), 123-168.

 \bibitem[G-K-K]{G-K-K} M. Golenishcheva-Kutuzova, V. Kac, $\Gamma$-conformal algebras, {\em J. Math. Phys.} {\bf 39} (1998),
2290-2305.




\bibitem[JKLT]{JKLT}
N. Jing, F. Kong, H.-S. Li, S. Tan, $(G, \chi_\phi)$-equivariant $\phi$-coordinated quasi modules for nonlocal vertex algebras, {\em J. Algebra} {\bf 570} (2021), 24-74.

\bibitem[Li1]{Li1}
H.-S. Li,  Local systems of vertex operators, vertex superalgebras
and modules,  {\em J. Pure Appl. Algebra} {\bf 109} (1996), 143-195.

\bibitem[Li2]{Li2}
H.-S. Li, $\phi$-Coordinated Quasi-Modules for Quantum Vertex Algebras, {\em Comm. Math. Phys.}  {\bf 308} (2011), 703-741.

\bibitem[Li3]{Li3}H.-S. Li, On certain generalizations of twisted affine Lie algebras and
quasimodules for $\Gamma$-vertex algebras, {\em J. Pure Appl. Algebra} {\bf 209} (2007), 853-871.

\bibitem[Li4]{Li4}H.-S. Li, $G$-equivariant $\phi$-coordinated quasi modules for quantum vertex algebras, {\em J. Math. Phys.} {\bf 54} (2013), 1-26.

\bibitem[Li5]{Li5}H.-S. Li, A new construction of vertex algebras and quasi-modules for vertex algebras, {\em Adv. Math.}
{\bf 202} (2006), 232-286.

\bibitem[Li6]{Li6}H.-S. Li, Local systems of twisted vertex operators, vertex superalgebras and twisted modules, Contemporary Math., Vol. {\bf 193}, Amer. Math. Soc., Providence, 1996, 203-236.



\bibitem[K]{K}V. G. Kac, {\em Vertex Algebras for Beginners,} Second edition, University Lecture Series, Vol. {\bf 10}, Amer. Math. Soc., Providence, 1998.

\bibitem[P]{P}M. Primc, Vertex algebras generated by Lie algebras, {\em J. Pure Appl. Algebra} {\bf 135} (1999), 253-293.

\end{thebibliography}
\end{document}